\documentclass[a4paper, 12pt]{amsart}

\usepackage{graphicx} 
\usepackage{amsmath}
\usepackage{amsthm}
\usepackage{amsfonts}
\usepackage{amssymb}
\usepackage{amscd}
\usepackage{color}
\usepackage[linktocpage]{hyperref}

\theoremstyle{plain}
  \newtheorem{thm}{Theorem}[section]
  \newtheorem*{thm*}{Theorem}
  \newtheorem{prop}[thm]{Proposition}
  \newtheorem{lem}[thm]{Lemma}
  \newtheorem{cor}[thm]{Corollary}
  \newtheorem{conj}[thm]{Conjecture}
\theoremstyle{definition}
  \newtheorem{dfn}[thm]{Definition}
  \newtheorem{exmp}[thm]{Example}
\theoremstyle{remark}
  \newtheorem{rem}[thm]{Remark}

\def\ZZ{\mathbb Z}
\def\QQ{\mathbb Q}

\def\CC{\mathbb C}

\def\cS{\mathcal{S}}
\def\cR{\mathcal{R}}

\def\0{\mathbf 0}

\def\Arf{\mathsf{Arf}}

\def\cl{\mathsf{cl}}
\def\pal{\mathsf{pal}}
\def\val{\mathsf{val}}

\def\<{\langle}
\def\>{\rangle}
\def\too{\longrightarrow}
\def\oot{\longleftarrow}
\def\one{\overline{1}}
\def\zero{\overline{0}}
\def\qb{\overline{q}}

\numberwithin{equation}{section} 

\date{}

\begin{document}

\title[Arithmetic on  $q$-deformed rational numbers]
{Arithmetic on $q$-deformed rational numbers}

\author[T. Kogiso]{Takeyoshi Kogiso}
\address{Department of Mathematics, Josai University, Saitama, 350-0295, Japan.} 
\email{kogiso@josai.ac.jp}

\author[K. Miyamoto]{Kengo Miyamoto}
\address{Department of Computer and Information Science, Ibaraki University, Ibaraki, 316-8511, Japan.}
\email{kengo.miyamoto.uz63@vc.ibaraki.ac.jp}

\author[X. Ren]{Xin Ren}
\address{Department of Mathematics, Kansai University, Osaka, 564-8680, Japan.}
\email{k641241@kansai-u.ac.jp}

\author[M. Wakui]{Michihisa Wakui} 
\address{Department of Mathematics,
Kansai University, Osaka, 564-8680, Japan.}
\email{wakui@kansai-u.ac.jp}

\author[K. Yanagawa]{Kohji Yanagawa} 
\address{Department of Mathematics, 
Kansai University, Osaka, 564-8680, Japan.}
\email{yanagawa@kansai-u.ac.jp}

\thanks{T. Kogiso is supported by JSPS Grant-in-Aid for Scientific Research (C) 21K03169.}
\thanks{K. Miyamoto is supported by JSPS Grant-in-Aid for Early-Career Scientists 20K14302, 24K16885 and	Grant-in-Aid for Scientific Research (A) 23H00479.}
\thanks{X. Ren is supported by the Kansai University Grant-in-Aid for progress of research in graduate course, 2023.}
\thanks{M. Wakui is supported by the Kansai University Grant-in-Aid for progress of research in graduate course, 2023.}
\thanks{K. Yanagawa is supported by JSPS Grant-in-Aid for Scientific Research (C) 19K03456 and by the Kansai University Grant-in-Aid for progress of research in graduate course, 2023.}

\begin{abstract}
Recently, Morier-Genoud and Ovsienko introduced a $q$-{deformation} of rational numbers. 
More precisely, for an irreducible fraction $\frac{r}s>0$, they constructed coprime polynomials $\cR_{\frac{r}s}(q), ~\cS_{\frac{r}s}(q) \in \ZZ[q]$ with  $\cR_{\frac{r}s}(1)=r, \cS_{\frac{r}s}(1)=s$. 
Their theory has a rich background and many applications. 
By definition, if $r \equiv r' \pmod{s}$, then $\cS_{\frac{r}s}(q)=\cS_{\frac{r'}s}(q)$. 
We show that $rr'{\equiv} -1 \pmod{s}$ implies $\cS_{\frac{r}s}(q)=\cS_{\frac{r'}s}(q)$, and it is conjectured that the converse holds if $s$ is prime (and $r \not \equiv r' \pmod{s}$). 
We also show that $s$ is a multiple of 3 (resp. 4) if and only if $\cS_{\frac{r}s}(\zeta)=0$ for $\zeta=(-1+\sqrt{-3})/2$ (resp. $\zeta=i$). 
We give applications to the representation theory of quivers of type $A$ and the Jones polynomials of rational links. 
\end{abstract}

\maketitle 
\tableofcontents

\section{Introduction}
The $q$-deformation of a positive integer $n$, which is  given by 
$$[n]_q=\frac{1-q^n}{1-q}=q^{n-1}+q^{n-2}+\cdots +q+1,$$
is a very classical subject of mathematics. Recently, 
Morier-Genoud and Ovsienko \cite{MO} introduced the $q$-deformation $[\alpha]_q$ of a rational number $\alpha$ based on some combinatorial properties of rational numbers. 
They further extended this notion to arbitrary real numbers \cite{MO22} by some number-theoretic properties of irrational numbers. 
These works are related to many directions including Teichm\"uller spaces \cite{VL}, the 2-Calabi-Yau category of type $A_2$~\cite{BBL}, the Markov-Hurwitz approximation theory \cite{K, SM, LMOV, XR}, the modular group and Picard groups {\cite{LSM, MOV, O21}}, Jones polynomials of rational knots \cite{Kogiso-Wakui, LeeSchiffler, NT, MO, BBL, Ren22}, and combinatorics on fence posets \cite{TBEC21, OEK1, OR}.

For an irreducible fraction $\frac{r}s >0$,  we have 
$$\left[\dfrac{r}{s}\right]_q =\dfrac{\cR_{\frac{r}s}(q)}{\cS_{\frac{r}s}(q)} \ \  \text{for} \ \ \cR_{\frac{r}s}(q), \cS_{\frac{r}s}(q)\in \ZZ_{>0}[q] \  \ \text{with} \ \ \cR_{\frac{r}s}(1)=r \ \  \text{and}  \ \ \cS_{\frac{r}s}(1)=s.$$ 
There are many ways to compute $[\alpha]_q$ (see Section \ref{sec-2} for details). 
For example, we have  
\[
\left[\frac{6}{5}\right]_q=\frac{[6]_q}{[5]_q}=\frac {q^5+q^4+q^3+q^2+q+1}{q^4+q^3+q^2+q+1}
, \quad
\left[\frac{7}{5}\right]_q
=\frac {{q}^{4}+2\,{q}^{3}+2\,{q}^{2}+q+1}{{q}^{3}+2\,{q}^{2}+q+1},
\]
\noindent 
and observe that the denominators of  $\frac{6}{5}$ and $\frac{7}{5}$ are the same 5, but the denominator polynomials of their $q$-deformation are different. 
In general, the following problem arises. When dose the  equation $\cS_{\frac{r}s}(q)=\cS_{\frac{r'}s}(q)$ hold for two irreducible fractions $\frac{r}s$ and $\frac{r'}s$? By definition, we have $\cS_{\alpha+n}(q)=\cS_{\alpha}(q)$ for $n \in \ZZ$, and hence $r \equiv r' \pmod{s}$ implies $\cS_{\frac{r}s}(q)=\cS_{\frac{r'}s}(q)$. However, there are more subtle relations.

\begin{exmp}
(1) The table of $\cS_\alpha(q)$ for irreducible fractions $\alpha$ of the form $\frac{r}{17}$ is the following.  
\begin{align*}
A&=[17]_q={q}^{16}+{q}^{15}+\cdots +q+1
\\
B&={q}^{9}+2\,{q}^{8}+2\,{q}^{7}+
2\,{q}^{6}+2\,{q}^{5}+2\,{q}^{4}+2\,{q}^{3}+2\,{q}^{2}+q+1
\\
C&={q}^{7}+2\,{q}^{6}+3\,{q}^{5}+3\,{q}^{4}
+3\,{q}^{3}+2\,{q}^{2}+2\,q+1
\\
D&={q}^{7}+2\,{q}^{6}+3\,{q}^{5}+4\,{q}^{4}
+3\,{q}^{3}+2\,{q}^{2}+q+1
\\
E&={q}^{6}+2\,{q}^{5}+4\,{q}^{4}+4\,{q}^{3}+
3\,{q}^{2}+2\,q+1
\end{align*}

\smallskip

\begin{center}
{\renewcommand{\arraystretch}{1.5}
\begin{tabular}{|c|c|c|c|c|c|c|c|c|c|}
\hline
$r \pmod{17}$&
$1,16$
&$2,8$
&$3,11$
&$4$
&$5,10$
&$6,14$
&$7,12$
&$9,15$
&13\\
\hline
$\mathcal{S}_{\frac{r}{17}}(q)$
&$A$&$B$&$C$&$D$&$E$&$C^\vee$&$E^\vee$&$B^\vee$&$D^\vee$\\
\hline
\end{tabular}
}
\end{center}

\smallskip

\noindent Here, for $f(q)\in {\mathbb Q}[q]$,   
$f^{\vee}(q)$ denotes its reciprocal polynomial $q^{\mathsf{deg}(f)} f(q^{-1})$.
For example, we have  
$$E^\vee={q}^{6}+2\,{q}^{5}+3\,{q}^{4}+4\,{q}^{3}+4\,{q}^{2}+
2\,q+1.$$

\smallskip

(2) Next, we give the table of $\cS_\alpha(q)$ for irreducible fractions $\alpha$ of the form $\frac{r}{23}$.  
\begin{align*}
A&=[23]_q=q^{22}+q^{21}+\cdots + q+1\\
B&= {q}^{12}+2\,{q}^{11}+2\,{q}^{10}+2\,{q}^{9}+2\,{q}^{8}+2\,{q}^{7}+2\,{q}^{6}+2\,{q}^{5}+2\,{q}^{4}+2\,{q}^{3}+2\,{q}^{2}+q+1\\
C&={q}^{9}+2\,{q}^{8}+3\,{q}^{7}+3\,{q}^{6}+3\,{q}^{5}+3\,{q}^{4}+3\,{q}^{3}+2\,{q}^{2}+2\,q+1\\
D&= {q}^{8}+2\,{q}^{7}+3\,{q}^{6}+4\,{q}^{5}+4\,{q}^{4}+3\,{q}^{3}+3\,{q}^{2}+2\,q+1\\
E&=  {q}^{7}+3\,{q}^{6}+4\,{q}^{5}+5\,{q}^{4}+4\,{q}^{3}+3\,{q}^{2}+2\,q+1\\
F&= {q}^{7}+2\,{q}^{6}+4\,{q}^{5}+5\,{q}^{4}+5\,{q}^{3}+3\,{q}^{2}+2\,q+1 
\end{align*}

\smallskip

\begin{center}
{\renewcommand{\arraystretch}{1.5}
\begin{tabular}{|c|c|c|c|c|c|c|c|c|c|c|c|}
\hline
\tiny{$r \pmod{23}$}
&$1,22$
&$2,11$
&$3,15$
&$4,17$
&$5,9$
&$6,19$
&$7,13$
&$8,20$
&$10,16$
&$14,18$
&$12,21$\\
\hline
$\mathcal{S}_{\frac{r}{23}}(q)$ 
&$A$
&$B$
&$C$
&$D$
&$E$
&$D^\vee$
&$F$
&$C^\vee$
&$F^\vee$
&$E^\vee$
&$B^\vee$
\\
\hline
\end{tabular}
}
\end{center}
\end{exmp}

From these examples, the third author of the present paper and Takeshi Sakurai, who were supervised by the first author, proposed the following conjecture in their master theses \cite{Ren21, Sa21}. This is the main motivation of the present paper.

\begin{conj}[Arithmetic conjecture]
\label{1-7}
Let $p$ be an odd prime integer. 
For two positive integers $a, b$ which are coprime to $p$, 
$\mathcal{S}_{\frac{a}{p}}(q)=\mathcal{S}_{\frac{b}{p}}(q)$ if and only if  
$ab\equiv -1 \pmod{p}$ or $a\equiv b \pmod{p}$. 
\end{conj}

The necessity part of Conjecture~\ref{1-7} really requires the assumption that $p$ is prime.  In fact, $\cS_{\frac{5}{24}}(q)=\cS_{\frac{11}{24}}(q)$ holds, while $5 \cdot 11 \not \equiv -1 \pmod{24}$. See Subsection~\ref{subsec-2-2} for detail. 
On the other hand, without the assumption that $p$ is prime, we can show the sufficiency (so the essential part of the conjecture is its necessity). We give two different proofs in Sections \ref{sec-3} and \ref{sec-4}. 

The proof given in Section~\ref{sec-3} is rather direct. 
Combining an argument here and a combinatorial result in \cite{OR}, we can show that $\cS_{\frac{r}s}(q)$ is palindromic if and only if $r^2 \equiv 1 \pmod{s}$. Recall that  $f(q) \in \ZZ[q]$ is said to be palindromic, if $f^\vee(q)=f(q)$.  

The proof given in Section~\ref{sec-4} uses the $q$-deformation $(a,b)_p \in \ZZ[q]$ of a pair $(a,b)$ of positive and coprime integers introduced in the previous work \cite{Wakui_LDTandNT} of the fourth author.
In Section~\ref{sec-5} 
we study on behavior of $\mathcal{R}_{\alpha}(q)$ and $\mathcal{S}_{\alpha}(q)$ under the operations $\mathfrak{i}, \mathfrak{r}, \mathfrak{ir}$ on the positive rational numbers $\alpha$, which are introduced in \cite{Kogiso-Wakui}. 

For a given rational number $\alpha \in \QQ \cap (1, \infty)$, the regular continued fraction expansion of $\alpha$ determines a quiver $Q$ of type $A$.
In \cite[Thoerem 4]{MO}, they provided a method for computing $\mathcal{R}_\alpha(q)$ (and $\mathcal{S}_\alpha(q)$) by using combinatorial enumeration with the quiver $Q$.
Specifically, the coefficients of $q^{k}$ in $\mathcal{R}_{\alpha}(q)$ coincides with the number of marking of circles to $k$ vertices of $Q$ so that there is no arrow from an unmarked vertex to a marked vertex.
Thus, one representation-theoretic view of $\mathcal{R}_{\alpha}(q)$ is that it counts the number of submodules of the most dimensional indecomposable module $M$ over the path algebra $\mathsf{k}Q$, where $\mathsf{k}$ is a field.
Namely, the coefficients of $q^{k}$ in $\mathcal{R}_{\alpha}(q)$ is equal to the number of $k$-dimensional submodules of $M$.
In Section \ref{sec-6}, we give a formula for computing $\mathcal{R}_\alpha(q)$.

In Section \ref{sec-7}, we extend the result \cite[Proposition~1.8]{MO} which states that $\cS_\alpha(-1)$ and $\cR_\alpha(-1)$ belong to  $\{0, \pm 1\}$. First, we will show that 
$$\cR_\alpha(\omega), \, \cS_\alpha(\omega) \in \{0, \pm 1, \pm \omega, \pm \omega^2 \} \quad \text{for} \quad \omega =\frac{-1+\sqrt{-3}}2$$
and 
$$\cR_\alpha(i), \, \cS_\alpha(i) \in \{0, \pm 1, \pm i, \pm(1+i), \pm (1-i)\}.$$
Hence, for an irreducible fraction $\frac{r}s$, $\cS_{\frac{r}s}(q) \in \ZZ[q]$ can be divided by $[3]_q=q^2+q+1$ (resp. $[4]_q=q^3+q^2+q+1$) if and only if $s$ is a multiple of 3 (resp. 4). Inspired by this fact, we conjecture that if $p$ is a prime integer then $\cS_{\frac{a}{p}}(q) \in \ZZ[q]$ is irreducible over $\QQ$ (Conjecture~\ref{irreducibility}).

In Section \ref{sec-8}, we give an application of the observations in the previous section. For the rational link $L(\alpha)$ associated with $\alpha \in \QQ$ (for example, see \cite{KL}), the Jones polynomial $V_{L(\alpha)}(t) \in \ZZ[t^{\pm1}] \cup t^{\frac12}\ZZ[t^{\pm 1}]$ has the normalized form $J_\alpha(q) \in \ZZ[q]$ (\cite{LeeSchiffler}). Since $J_{\alpha }(q)$ for $\alpha >1$ can be expressed using $\mathcal{R}_{\alpha }(q)$ and $\mathcal{S}_{\alpha }(q)$ 
 by \cite[Proposition A.1]{MO}, one can study the special values of $J_\alpha(q)$ at $q=-1, i, \pm \omega$. There are several classical results on the special values of the Jones polynomials $V_L(t)$ for general links $L$, and most of the facts given in this section easily follow from these results. However, we give a new explanation using $q$-deformed rationals.

\section*{Acknowledgments} 
{We would like to thank Professor Mikami Hirasawa for helpful comments and references on results on existence of the Arf invariants of rational links. 
We would also like to thank Professors Sophie Morier-Genoud, Valentin Ovsienko,  and Taizo Kanenobu for their encouragements of our research. 
We are grateful to the referee for useful suggestions for improving the manuscript. }

\section{Preliminaries}\label{sec-2}
Throughout this paper, for a real number $x\in \mathbb{R}$, the symbols $\lceil x \rceil$ and $\lfloor x \rfloor$ mean
$\lceil x \rceil=\min\{n\in\mathbb{Z}\mid x\leq n\}$ and
$\lfloor x \rfloor=\max\{n\in\mathbb{Z}\mid n\leq x\}$, respectively.
For an irreducible fraction $\frac{r}s$, we always assume that $s >0$. We regard $0 =\frac{0}{1}$ as an irreducible fraction.

\subsection{\texorpdfstring{$q$}{q}-deformed rational numbers}
\label{subsec-2-1}

In this subsection,  we review some basics on $q$-deformations for rational numbers introduced by Morier-Genoud and Ovsienko {\cite{MO,MOV}}.
A rational number $\alpha\in \mathbb{Q}\cap ({1},\infty)$ can be represented by 
\[ \alpha=a_1+\dfrac{1}{a_2+\dfrac{1}{\raisebox{0.5cm}{$\ddots$}+\dfrac{1}{a_{n}}}}\]
with 
${a_1},\ldots, a_{n}\in\mathbb{Z}_{>0}$ 
{and} it  can be also represented by
\[ \alpha=c_1-\dfrac{1}{c_2-\dfrac{1}{\raisebox{0.5cm}{$\ddots$}-\dfrac{1}{c_{l}}}}\]
with $c_1,\ldots, c_{l}\in\mathbb{Z}_{> {1}}$.
In this case, we write $[a_1,\ldots, a_{n}]$ and $[[c_1,\ldots, c_l]]$ for these expansions, respectively. 
The former expansion is called a \textit{regular continued fraction} of $\alpha$, and the latter is called a {Hirzebruch-Jung continued fraction (or \textit{negative continued fraction} in this paper)} of $\alpha$.
One can always assume that the length $n$ of a regular continued fraction to be even, since  $[a_1,\ldots, a_{n}+1]=[a_1,\ldots, a_{n},1]$. The expression as a regular continued fraction is uniquely determined if the parity of $n$ is specified, and that as a negative continued fraction is unique (since $c_i \ge 2$ for all $i$ now).

For an integer $a$, we set:
\begin{equation}\label{mat. cont. farac.} M(a):=\begin{pmatrix}
a & 1\\
1 & 0
\end{pmatrix}, \quad M^{-}(a):=\begin{pmatrix}
a & -1\\
1 & 0
\end{pmatrix}. \end{equation}
Moreover, for a finite sequence of integers $(a_1,\ldots, a_n)$, we set
\begin{equation}\label{mat. cont. frac. seq.}
M(a_1,\ldots, a_n)=M(a_1)\cdots M(a_n),\quad M^{-}(a_1,\ldots, a_n)=M^{-}(a_1)\cdots M^{-}(a_n).
\end{equation}
It follows from the definitions, we see that $M^{-}(a_1,\ldots, a_n)\in\mathsf{SL}(2,\mathbb{Z})$, whereas $M(a_1,\ldots, a_n)\in\mathsf{SL}(2,\mathbb{Z})$ if and only if $n$ is even.
These matrices are well-known as the matrices of continued fractions in elementary number theory because one has the following result.

\begin{lem}[{\cite[Proposition 3.1]{MO19}}]\label{MO19 prop.3.1}
    Let $\alpha=\frac{r}{s}>1$ be an irreducible fraction, and assume that it is expressed by 
    \[ \alpha= [a_1,\ldots, a_{{n}}]=[[c_1,\ldots, c_l]]\]
    with $a_i\geq 1$ $(i=1,\ldots , {n} )$ and $c_j\geq 2$ $(j=1,\ldots ,l)$.
    Then, 
    \[ M(a_1,\ldots, a_{{n}})=\begin{pmatrix}
r & r'\\
s & s'
\end{pmatrix},\quad M^{-}(c_1,\ldots, c_l)=\begin{pmatrix}
r & -r''\\
s & -s''
\end{pmatrix}, \]
where $\frac{r'}{s'}=[a_1,\ldots,a_{{n-1}}]$ and $\frac{r''}{s''}=[[c_1,\ldots,c_{l-1}]]$.
\end{lem}

The $q$-deformation of positive rational numbers is based on the above lemma.
Let $q$ be a formal symbol.
For an integer $a$, we define a Laurent polynomial $[a]_q\in \mathbb{Z}[q, q^{-1}]$ by
\[ [a]_q:= \dfrac{1-q^a}{1-q}=\left\{\begin{array}{ll}
     q^{a-1}+q^{a-2}+\cdots +q+1& \text{if $a>0$},  \\[5pt]
     0 & \text{if $a=0$}, \\[5pt]
    -q^{-a}-q^{-a+1}-\cdots -q^{-2}-q^{-1}& \text{if $a<0$}.
\end{array}\right.\]
By the definition of $[a]_q$, for all $a,n\in\mathbb{Z}$, the equation
\begin{equation}\label{[a+n]_q}
[a+n]_q=q^n[a]_q+[n]_q 
\end{equation}
holds.
For an integer $a$, two $q$-deformations of (\ref{mat. cont. farac.}) are defined by
\begin{equation}\label{q-deform. mat. cont. farac.} 
M_q(a):=\begin{pmatrix}
[a]_q & q^a\\
1 & 0
\end{pmatrix}, \quad 
M^{-}_q(a):=\begin{pmatrix}
[a]_q & -q^{a-1}\\
1 & 0
\end{pmatrix}.
\end{equation}
{
The next lemma is a $q$-deformation of Lemma \ref{MO19 prop.3.1}. 
Here, for regular continued fractions, we only use those of even length.  
The $q$-deformations of \eqref{mat. cont. frac. seq.} are defined as follows. 
\begin{eqnarray*}\label{q-deform. mat. cont. frac. seq.}
M_q(a_1,\ldots ,a_{2m})  &:=& 
        M_q(a_1)M_{q^{-1}}(a_2)M_q(a_3)\cdots M_{q^{-1}}(a_{2m})\\
\widetilde{M}_q(a_1,\ldots, a_{2m}) &:=& q^{a_2+a_4+\cdots +a_{2m}}M_q(a_1,\ldots, a_{2m})\\
M^{-}_q(a_1,\ldots, a_n) &:=& M^{-}_q(a_1)M^{-}_q(a_2)\cdots M^{-}_q(a_n).
\end{eqnarray*} 
Then, the following statements hold.  
}

\begin{prop}[{\cite[Propositions 4.3 and 4.9]{MO}}]\label{MO20 prop. 4.3 and 4.9}
Let $\alpha=\frac{r}{s}$ be a rational number as given in Lemma \ref{MO19 prop.3.1}. {The polynomials $\cR_{\alpha}(q), \cS_{\alpha}(q) \in \ZZ[q]$ given by  
\[ 
M_q^{-}(c_1,\ldots , c_l)\begin{pmatrix}
 1 \\ 0 
\end{pmatrix}=\begin{pmatrix}
 \cR_{\alpha}(q) \\ \cS_{\alpha}(q)
\end{pmatrix}
\]
also satisfy 
\begin{equation}\label{q-deformation via regular cf}
\widetilde{M}_q(a_1,\ldots , a_{2m})\begin{pmatrix}
 1 \\ 0 
\end{pmatrix}=\begin{pmatrix}
 q\cR_{\alpha}(q) \\ q\cS_{\alpha}(q)
\end{pmatrix}. 
\end{equation}
Moreover, the following statements hold.}
\begin{enumerate}
    \item[(1)] $\cR_{\alpha}(q)$ and $\cS_{\alpha}(q)$ are coprime in $\mathbb{Z}[q]$. 
    \item[(2)] We have $\cR_{\frac{r}s}(1)=r$ and $\cS_{\frac{r}s}(1)=s$. 
\end{enumerate}
\end{prop}

Based on Proposition \ref{MO20 prop. 4.3 and 4.9}, the \textit{$q$-deformation of a rational number $\alpha>1$} is defined by
\[ [\alpha]_q: =\dfrac{\cR_\alpha(q)}{\cS_\alpha(q)}.\]

{
\begin{rem}
Let $\mathsf{PSL}_{q}(2,\mathbb{Z})$ be the subgroup of 
$$ \mathsf{PGL}\left(2,\mathbb{Z}\left[q^{\pm 1}\right]\right)=\mathsf{GL}\left(2,\mathbb{Z}\left[q^{\pm 1}\right]\right)/\left\{\pm q^NE_2\mid N\in \mathbb{Z}\right\}$$ generated by the following two matrices 
\[\displaystyle
R_q:=\begin{pmatrix}
q & 1 \\
0 & 1 \\
\end{pmatrix}, \ \ \ \ 
L_q = \begin{pmatrix}
1 & 0 \\
1 & q^{-1} \\
\end{pmatrix}. 
\] 
 \cite[Proposition 1.1]{LSM} states that  $\mathsf{PSL}(2,\mathbb{Z}) \cong \mathsf{PSL}_{q}(2,\mathbb{Z})$. Via the equation 
 $$M_q(a_1,\ldots, a_{2m})=R_q^{a_1}L_q^{a_2}R_q^{a_3}L_q^{a_4}\cdots R_q^{a_{2m-1}} L_q^{a_{2m}}$$
and the classical $\mathsf{PSL}(2,\mathbb{Z})$ 
 action on $\QQ \cup \left\{\left(\frac{1}{0}\right)\right\}$, 
\cite{MOV} gives an insightful interpretation of $q$-deformed rationals. 
We can also use negative continued fractions for this interpretation. 
\end{rem}
}

For an integer $n \ge 2$, since $n=[[n]]$ as a negative continued fraction, we have the following 
philosophically trivial equations 
\begin{equation}\label{[n]_q}
\cR_{n}(q)=[n]_q \quad \text{and} \quad \cS_{n}(q)=1.
\end{equation}

Morier-Genoud and Ovsienko pointed out that the definition of $q$-deformed rational number $[\alpha ]_q$ can be extended to the case where $\alpha \leq 1$ including the negative rational numbers by the following formulas, see \cite[page 3]{MO}:
\begin{equation}
[\alpha +1]_q=q[\alpha ]_q+1. \label{eq1-16}
\end{equation}
However, for $\alpha<0$, $\cR_{\alpha}(q)$ is not an ordinary polynomial but a Laurent polynomial. Similarly, for $0 < \alpha<1$, $\cR_{\alpha}(q)$ is a polynomial, but $\cR_{\alpha}(0)=0$ (if $\alpha \ge 1$, we have $\cR_{\alpha}(0)=1$).   
{ It can be easily verified that \eqref{q-deformation via regular cf} holds for all $\alpha \in \QQ$, that is, without assuming that $\alpha >1$. }

{
\begin{lem}\label{wakui-1.5}
For a rational number $\alpha $ and an integer $n$, we have
$$
\mathcal{R}_{\alpha +n}(q)=q^n\mathcal{R}_{\alpha }(q) +[n]_q\mathcal{S}_{\alpha }(q) \quad \text{and} \quad \mathcal{S}_{\alpha +n}(q)=\mathcal{S}_{\alpha }(q),
$$
equivalently, $[\alpha +n]_q=q^n[\alpha ]_q+[n]_q.$
{In particular, we have
\begin{equation}\label{right-sift}
\cS_{\alpha}(q)=\cS_{\alpha+1}(q) \quad \text{and} \quad 
\cR_{\alpha}(q)=q^{-1}(\cR_{\alpha+1}(q) -\cS_{\alpha+1}(q)).
\end{equation}
}
\end{lem} 
\begin{proof} 
It suffices to show that $[\alpha +n]_q=q^n[\alpha ]_q+[n]_q$. For $n \geq 1$, this is easily shown by induction on $n$ using \eqref{eq1-16}. 
For $n \ge 1$, replacing $\alpha$ by $\alpha -n$, we have $[\alpha]_q=q^n[\alpha-n]_q+[n]_q$. Hence 
$$[\alpha-n]_q= q^{-n}[\alpha]_q-q^{-n}[n]_q= q^{-n}[\alpha]_q+[-n]_q.$$
\end{proof} 
}

\begin{lem}\label{wakui-1.4}
Let $a, x$ be positive and coprime integers with $1\leq a\leq x$, and express $x$ as the form 
$x=ca+r$ for some  $c, r\in\mathbb{Z}$ with $0\leq r<a$. 
Then the following equations hold: 
\begin{align*}
\mathcal{R}_{\frac{x}{a}}(q)&=[c+1]_q\mathcal{R}_{\frac{a}{a-r}}(q)-q^{c}\mathcal{S}_{\frac{a}{a-r}}(q), \\[0.1cm]  
\mathcal{S}_{\frac{x}{a}}(q)&=\mathcal{R}_{\frac{a}{a-r}}(q).   
\end{align*}
\end{lem} 
\begin{proof}
Note that it follows from the equations \eqref{[n]_q} and \eqref{right-sift} that $\cR_1(q)=\cS_1(q)=1$.
If $a=1$, then $r=0$.
Thus, we have
\[ 
[c+1]_q\mathcal{R}_{\frac{a}{a-r}}(q)-q^{c}\mathcal{S}_{\frac{a}{a-r}}(q)=[x+1]_q\cR_1(q)-q^x\cS_1(q) = [x]_q =\mathcal{R}_{\frac{x}{a}}(q).
\]
The second equation obviously holds when $a=1$. 

If $a>1$, then $r>0$, and thus $\frac{a}{a-r}>1$.
By $x=ca+r$, 
\[ 
\dfrac{x}{a}=\dfrac{(c+1)a+r-a}{a}=c+1-\dfrac{1}{\frac{a}{a-r}}.
\]
So,  if $\frac{a}{a-r}$ is expressed as  
$\frac{a}{a-r}=[[c_1, \ldots , c_l]]$, then 
$\frac{x}{a}=[[c+1, c_1, \ldots ,c_l]]$ and 
\begin{align*}
M_q^-(c+1, c_1, \ldots , c_l)
\begin{pmatrix} 
1\\ 0 \end{pmatrix}
&=M_q^-(c+1)M_q^-(c_1, \ldots , c_l)\begin{pmatrix} 
1\\ 0 \end{pmatrix} \\[0.1cm]   
&=\begin{pmatrix} 
[c+1]_q \mathcal{R}_{\frac{a}{a-r}}(q)-q^{c}\mathcal{S}_{\frac{a}{a-r}}(q) \\[0.1cm]  
\mathcal{R}_{\frac{a}{a-r}}(q)
\end{pmatrix} . 
\end{align*}
This leads to the equations in the lemma. 
\end{proof}

{By Lemmas \ref{wakui-1.5} and} \ref{wakui-1.4}, we have 
$$\{ \, \cS_\alpha(q) \mid  \alpha \in \QQ \, \} 
{=\{ \, \cS_\alpha(q) \mid  \alpha \in \QQ  \cap (1, 2] \, \} }
= \{ \, \cR_\alpha (q) \mid  \alpha \in \QQ \cap (1, \infty) \, \}. $$

\begin{lem}\label{wakui-1.6}
For coprime positive integers $a, x$ with $1\leq a\leq x$, we have
\begin{align}
\mathcal{R}_{\frac{a}{x}}(q) &=\mathcal{R}_{\frac{x}{x-a}}(q) -\mathcal{S}_{\frac{x}{x-a}}(q), \label{wakui-1.6-1}\\
\mathcal{S}_{\frac{a}{x}}(q) &=\mathcal{R}_{\frac{x}{x-a}}(q).   \label{wakui-1.6-2}
\end{align}
\end{lem} 
\begin{proof}
Express $\frac{x}{x-a}$ as the negative continued fraction $\frac{x}{x-a}=[[c_1, \ldots , c_l]]$. 
Then $\frac{a}{x}+1=[[2, c_1, \ldots , c_l]]$, and  
$$\begin{pmatrix}
\mathcal{R}_{\frac{a}{x}+1}(q) \\ 
\mathcal{S}_{\frac{a}{x}+1}(q)
\end{pmatrix} 
=M_q^-(2)M_q^-(c_1, \ldots , c_l)\begin{pmatrix} 
1 \\ 0 \end{pmatrix}
=\begin{pmatrix} 
[2]_q & -q \\ 
1 & 0 
\end{pmatrix}\begin{pmatrix}
\mathcal{R}_{\frac{x}{x-a}}(q) \\  
\mathcal{S}_{\frac{x}{x-a}}(q)
\end{pmatrix} .$$
This equation and Lemma \ref{wakui-1.5} yield the equation \eqref{wakui-1.6-2} and
\begin{equation}\label{wakui-1.6-3}
q\mathcal{R}_{\frac{a}{x}}(q) +\mathcal{S}_{\frac{a}{x}}(q)
=[2]_q\mathcal{R}_{\frac{x}{x-a}}(q)-q\mathcal{S}_{\frac{x}{x-a}}(q) 
\end{equation}
The equation \eqref{wakui-1.6-1} can be obtained by substituting \eqref{wakui-1.6-2} to \eqref{wakui-1.6-3}. 
\end{proof}

\subsection{The arithmetic conjecture on \texorpdfstring{$q$}{q}-deformed rational numbers}\label{subsec-2-2}
Conjecture \ref{1-7} is the central problem of the present paper.
In this subsection, we collect a few remarks on this conjecture. 

{
If Conjecture~\ref{1-7} holds for an odd prime $p$, then we have   
\begin{equation}\label{number of S}
\#\{ \mathcal{S}_{\frac{a}{p}}(q) \mid a \in \ZZ \} = \begin{cases}
\dfrac{p+1}{2} & (p\equiv 1\pmod4), \\[5mm]
\dfrac{p-1}{2} & (p\equiv 3\pmod4).
\end{cases}
\end{equation}
To see this, recall the result of elementary number theory that 
there is some $a \in \ZZ$ with $a^2 \equiv -1 \pmod{p}$ if and only if $p \equiv 1 \pmod4$. 
Thus, if $p\equiv 1\pmod 4$, then 
$$\{1, \ldots, p-1\}=\{\, a_1, \ldots, a_{\frac{p-3}2}, b_1, \ldots, b_{\frac{p-3}2}, c, d \, \},$$
where $a_ib_i \equiv -1 \pmod{p}$ for each $i$ and $c^2 \equiv d^2 \equiv -1 \pmod{p}$. If $p\equiv 3\pmod4$, $$\{1, \ldots, p-1\}=\{\, a_1, \ldots, a_{\frac{p-1}2}, b_1, \ldots, b_{\frac{p-1}2} \, \}$$
holds, where $a_ib_i \equiv -1 \pmod{p}$ for each $i$.  In the present assumption, 
we have $\mathcal{S}_{\frac{a_i}p}(q)=\mathcal{S}_{\frac{b_i}p}(q)$ for each $i$, and this is the only case when $\mathcal{S}_{\frac{a}p}(q)=\mathcal{S}_{\frac{b}p}(q)$ holds for distinct $a,b \in \{1, \ldots ,  p-1\}$. Hence Conjecture~\ref{1-7} implies \eqref{number of S}. 
However, in Theorem~\ref{sufficiency} below, we will prove the sufficiency of the conjecture 
(without assuming that $p$ is prime). 
So $\mathcal{S}_{\frac{a_i}p}(q)=\mathcal{S}_{\frac{b_i}p}(q)$ actually holds, and \eqref{number of S} is equivalent to Conjecture~\ref{1-7}. } 
\smallskip

Next, we remark that the assumption that $p$ is prime is really necessary for the necessity part of Conjecture~\ref{1-7}. 
In fact, $\frac{5}{24}=[0,4,1,4]$ and $\frac{11}{24}=[0,2,5,2]$ satisfy 
$$\cS_{\frac{5}{24}}(q)=\cS_{\frac{11}{24}}(q)=q^8+2q^7+3q^6+4q^5+4q^4+4q^3+3q^2+2q+1$$ 
by Proposition~\ref{MO20 prop. 4.3 and 4.9} and \eqref{right-sift}, while $5\cdot 11 +1=56$ is not divisible by 24.

The following table shows composite numbers $p$ and pairs of natural numbers $(a, b)$ $(1<a<b<p\leq 111)$ which do not satisfy the necessity of Conjecture~\ref{1-7}. Note that if $p$ admits a pair $(a,b)$ with this property then it admits other pairs. For example, $(p-b, p-a)$ is also such a pair by Lemma~\ref{-r} below.

\begin{center}
\begin{tabular}{cc|cc} \hline
   $p$ & $(a,b)$ & $p$ & $(a,b)$\\ \hline
   24 & (5,11) & 84 & (19,25)  \\ 
    51 & (11,20) & 91 & (19,32)  \\ 
    57 & (13,16) & 99& (17,28)  \\ 
    60 & (11,19) & 105& (23,38)  \\ 
    63 & (13,20) & 110& (19,41)  \\ 
    78 & (17,29) & 111 & (25,34)  \\  \hline
\end{tabular}
\end{center}

On the other hand, the sufficiency part of Conjecture~\ref{1-7} holds without the assumption that $p$ is prime. 
In Sections \ref{sec-3} and \ref{sec-4}, we will prove this in two ways.

\subsection{Closures of a quiver and  \texorpdfstring{$q$}{q}-deformed rational numbers}  \label{subsec2-3}

By a quiver we mean a tuple $Q=(Q_0,Q_1,s,t)$ consisting of two sets $Q_0$, $Q_1$ and two maps $s,t:Q_1\to Q_0$. Each element of $Q_0$ (resp. $Q_1$) is called a vertex (resp. an arrow). For an arrow $\alpha\in Q_1$, we call $s(\alpha)$ (resp. $t(\alpha)$) the source (resp. the target) of $\alpha$. We will commonly write $a\xrightarrow{\alpha}b$ or $\alpha:a\to b$ to indicate that an arrow $\alpha$ has the source $a$ and the target $b$.
A quiver $Q$ is \textit{finite} if two sets $Q_0$ and $Q_1$ are finite sets.
The \textit{opposite quiver} of $Q$, say $Q^\vee$, is defined by $Q^\vee=(Q_0,Q_1,t,s)$.

Let $Q$ be a finite quiver.
A subset $C\subset Q_0$ is a \textit{closure} if there is no arrow $\alpha\in Q_1$ such that $s(\alpha)\in Q_0\setminus C$ and $t(\alpha)\in C$.
A closure $C$ is an \textit{$\ell$-closure} if the number of elements of $C$ is $\ell$.
The number of $\ell$-closures is denoted by $\rho_\ell(Q)$.
Then the polynomial 
\[ \mathsf{cl}(Q):=\sum_{\ell=0}^{n}\rho_\ell(Q)q^{\ell}\in\mathbb{Z}[q],\]
where $n=|Q_0|$, is called the \textit{closure polynomial} of $Q$.

Obviously, the constant term and the coefficient of the leading term of $\cl(Q)$ are $1$, including the extremal case $\mathsf{cl}(\varnothing)=1$.
We remark that, for any $\ell$, the equation
\begin{equation}\label{opposite closure poly}
\rho_\ell(Q)=\rho_{n-\ell}(Q^\vee)
\end{equation}
holds.
For a polynomial $f(q)\in\mathbb{Z}[q]$, we define a polynomial $f^{\vee}(q)$ by 
\[ f^{\vee}(q)=q^{\mathsf{deg}(f)}f(q^{-1}),\]
which is called the \textit{reciprocal polynomial} of $f(q)$.
By \eqref{opposite closure poly}, we have 
\begin{equation} \label{opposite closure poly 2}
\cl(Q)^\vee=\cl(Q^\vee).    
\end{equation}

For a tuple of integers $\mathbf{a}:=(a_1,a_2,\ldots, a_s)$ with $a_1, a_s\geq 0$, $a_2,\ldots,a_{s-1} >0$, 
we set the quiver
\[ Q(\mathbf{a}):= \underbrace{\circ \oot \circ \cdots \circ \oot \circ}_{a_1 \, \text{left arrows}}\underbrace{ \too \circ \cdots \circ \too \circ}_{a_2  \, \text{right arrows}}\underbrace{\oot \circ \cdots \circ \oot \circ}_{a_3  \, \text{left  arrows}} \too \cdots,\]
with the left-right distinction.
We understand that if $a_1=0$, then 
\[ Q(\mathbf{a}):= \underbrace{ \too \circ \cdots \circ \too \circ}_{a_2  \, \text{right arrows}}\underbrace{\oot \circ \cdots \circ \oot \circ}_{a_3  \, \text{left arrows}} \too \cdots.\]
Note that $|Q(\mathbf{a})_0|=a_1+a_2+\cdots +a_s+1$, and, for $\mathbf{a}=(a_1,a_2,\ldots, a_s)$, the equation
\begin{equation}\label{(0,a)=a^v}
\cl(Q(0,\mathbf{a}))=\cl(Q(\mathbf{a}))^\vee
\end{equation}
holds since $Q(0,\mathbf{a})\simeq Q(\mathbf{a})^{\vee}$ as quivers. 
Here we have $\cl(Q(0,0))=\cl(Q(0))=1+q$.

{\begin{rem}
We note that the closure polynomial $\cl(Q(\mathbf{a}))$ of a quiver $Q(\mathbf{a})$ can be realized with the \textit{rank polynomials of a finite fence poset}, which is more common in combinatorics (see \cite{TBEC21} and \cite{OR} for detail). 
\end{rem}}

\begin{lem}\label{lem:closure of palindromic}
For $\mathbf{a}=(a_1,a_2,\ldots, a_{s})$, we put $\mathbf{a}^\pal:=(a_s,a_{s-1},\ldots, a_{1})$. Then, there is an isomorphism of quivers
\[
Q(\mathbf{a}^\pal)\simeq\left\{\begin{array}{ll}
     Q(\mathbf{a})& \text{if $s$ is even,} \\ [8pt]
     Q(\mathbf{a})^\vee& \text{ if $s$ is odd.} 
\end{array}\right.
\]
Therefore, we have 
\[
\cl(Q(\mathbf{a}^\pal))=\left\{\begin{array}{ll}
     \cl(Q(\mathbf{a})) & \text{if $s$ is even,} \\ [8pt]
     \cl(Q(\mathbf{a}))^\vee& \text{ if $s$ is odd.} 
\end{array}\right.
\]
\end{lem}
\begin{proof}
First, we assume that $s$ is even. 
Then, the direction of the $i$-th arrow of $Q(\mathbf{a})$ from the left is the opposite of that of the $i$-th arrow of $Q(\mathbf{a}^\pal)$ from the right end.
Thus, $Q(\mathbf{a}^\pal)$ is the \lq\lq$\pi$-rotation\rq\rq  of $Q(\mathbf{a})$, and hence $Q(\mathbf{a})\simeq Q(\mathbf{a}^\pal)$ as quivers.
We leave the case $n$ is odd to the reader as an easy exercise.
\end{proof}

According to \cite[Section 3]{MO}, Morier-Genoud and Ovsienko gave a combinatorial interpretation of the coefficients in $\cR_{\alpha }(q)$ and $\cS_{\alpha }(q)$. 

Let $\alpha >1$ be a rational number, and write $\alpha$ as the regular continued fraction $\alpha=[a_1, a_2, \ldots, a_{2m}]$.
Then, we set 
\begin{align*}\label{Q_a^R}
& Q_{\alpha}^{\cR} :=Q(a_1-1, a_2, \ldots , a_{2m-1}, a_{2m}-1), \\
& Q_{\alpha}^{\cS} :=\left\{\begin{array}{ll}
Q(0,a_2-1, a_3, \ldots , a_{2m-1}, a_{2m}-1)  & \text{if $m>1$,} \\ [5pt]
Q(0,a_2-2)  & \text{if $m=1$.}
\end{array}\right.
\end{align*}
Here, if $a_2=1$ and $m>1$ (resp. $a_2=2$ and $m=1$, $a_2=1$ and $m=1$), we understand that $Q_\alpha^{\mathcal{S}}=Q(a_3, \ldots , a_{2m-1}, a_{2m}-1)$ (resp. $Q_\alpha^{\mathcal{S}}=Q(0)$, $Q_\alpha^{\mathcal{S}}=\varnothing$).
The quiver $Q_{\alpha}^{\cS}$ is obtained by deleting the first $a_1$ arrows from $Q_\alpha^{\cR}$.

\begin{rem}\label{odd length}
 If $\alpha \not \in \ZZ$ and $\alpha >1$, the above construction of $Q_\alpha^\cR$ and $Q_\alpha^\cS$  also works for the expression as a regular continued fraction of {\it odd} length.
\end{rem}

Following the notation used in \cite{MO}, we will use the symbols $\rho_\ell(\alpha)$ and $\sigma_\ell(\alpha)$ to denote the numbers of $\ell$-closures of  $Q_\alpha^{\cR}$ and $Q_\alpha^{\cS}$, respectively.

\begin{thm}[{\cite[Theorem 4]{MO}}]\label{theorem:MO closure} 
Let $\alpha >1$ be an irreducible fraction.
Then, the following equations hold:
\begin{align}
\mathcal{R}_{\alpha}(q) &=\sum_{\ell\geq 0}\rho_\ell(\alpha) q^\ell\left(=\cl(Q_{\alpha}^{\cR})\right), \\
\mathcal{S}_{\alpha}(q) &=\sum_{\ell \geq 0}\sigma_\ell(\alpha) q^\ell\left(=\cl(Q_{\alpha}^{\cS})\right).
\end{align}
\end{thm} 

\subsection{Farey neighbors and Farey sums}
\label{Farey sumes and the SB tree}

In this subsection, we recall the definitions of Farey neighbors and Farey sums. 

Two irreducible fractions $\frac{x}{a}, \frac{y}{b}$ are said to be   
\textit{Farey neighbors} if $ay-bx=1$. 
Here we regard $\infty =\frac{1}{0}$ as an irreducible fraction.

For two irreducible fractions $\frac{x}{a}, \frac{y}{b}$, the operation $\sharp$ is defined as follows:
\[ \frac{x}{a}\sharp \frac{y}{b}:=\frac{x+y}{a+b}. \]
If $\frac{x}{a}, \frac{y}{b}$ are Farey neighbors, then $\frac{x}{a}\sharp \frac{y}{b}$ is called the \textit{Farey sum} of $\frac{x}{a}$ and $\frac{y}{b}$.
The Farey sum of two irreducible fractions is also irreducible. Farey neighbors have the following fundamental properties. 

\begin{lem}\label{Farey neighbors fandamentals}
The following assertions hold.
\begin{enumerate}
\item[$(1)$] Any non-negative rational number can be obtained from $\frac{0}{1}$ and $\frac{1}{0}$ applying  $\sharp$ in finitely many times. 
\item[$(2)$] For any positive rational number $\alpha \in (0,\infty)$, there uniquely exist Farey neighbors $\frac{x}{a}, \frac{y}{b}$ such that $\alpha =\frac{x}{a}\sharp \frac{y}{b}$. 
The pair $(\frac{x}{a}, \frac{y}{b})$ is called the \textit{Farey parent} of $\alpha$, and the fraction $\frac{x}{a}$ (resp. $\frac{y}{b}$) is called the left parent (resp. the right parent). 
\end{enumerate}
\end{lem} 
For proof of the above lemma, see \cite[Theorem 3.9]{Aigner} or \cite[Lemma 3.5]{Kogiso-Wakui_Proc}.

\medskip

Let $\alpha$ and $\beta$ be two fractions with $\alpha,\beta\geq 1$.
If $\alpha\sharp\beta=[[c_1,\ldots, c_{l}]]$, then the equation
\begin{equation}\label{q-deformed farey sum}
    [\alpha\sharp\beta]_q=\dfrac{\mathcal{R}_{\alpha}(q)+q^{c_l-1}\mathcal{R}_{\beta}(q)}{\mathcal{S}_{\alpha}(q)+q^{c_l-1}\mathcal{S}_{\beta}(q)}
\end{equation}
holds, see \cite[Theorem 3]{MO}.

\section{A proof of the sufficiency of the conjecture} \label{sec-3}

In this section, without the assumption that $p$ is a prime number, we will show that $ab \equiv -1 \pmod{p}$ implies 
$
\cS_{\frac{a}{p}}(q)=\cS_{\frac{b}{p}}(q),
$
that is, the sufficiency part of Conjecture \ref{1-7} holds. 
Recall that $\cS_{\alpha+n}(q)=\cS_{\alpha}(q)$ for all $\alpha \in \QQ$ and $n \in \mathbb{Z}$. 

In the rest of the paper, $p$ means a (not necessarily prime) integer with $p \ge 2$, unless otherwise specified.

\begin{lem}\label{-r}
Let $\frac{a}{p},  \frac{b}{p}$ be irreducible fractions  with $a \equiv -b \pmod{p}$. 
We may assume that $ \frac{a}{p} - \left \lfloor  \frac{a}{p} \right \rfloor \le  \frac{1}{2}$, and hence $\frac{a}{p}=[a_1, a_2, \ldots, a_k]$ with $a_2 \ge 2$ as a regular continued  fraction.  
Then we have $\frac{b}{p}=[b_1, 1, a_2-1, a_3,\ldots, a_k]$, where $b_1 =\left\lfloor  \frac{b}{p} \right\rfloor$. 
\end{lem}

\begin{proof}
Since the assertion only depends on the decimal parts of $\frac{a}{p}$ and $\frac{b}{p}$, we may assume that $0< \frac{a}{p} \le \frac{b}{p} <1$. 
Then we have $b=p-a$, 
$$\frac{a}{p}=\cfrac{1}{\cfrac{p}{a}}=\cfrac{1}{a_2+\cfrac{p-aa_2}{a}}$$ 
and 
$$
\frac{b}{p}=\cfrac{p-a}{p}=\frac{1}{\cfrac{p}{p-a}}=\cfrac{1}{1+\cfrac{a}{p-a}}=\cfrac{1}{1+\cfrac{1}{\cfrac{p-a}{a}}}=\cfrac{1}{1+\cfrac{1}{(a_2-1)+\cfrac{p-aa_2}{a}}}. $$
If $k >2$, we have  
$\dfrac{p-a a_2}{a}=[a_3, a_4, \ldots, a_k]$, and the assertion follows. 
\end{proof}

\begin{prop}\label{dual1}
Let $\frac{a}{p},  \frac{b}{p}$ be irreducible fractions with $a \equiv -b \pmod{p}$. Then we have $\cS_{\frac{a}{p}}(q)=\cS^\vee_{\frac{b}{p}}(q)$. 
\end{prop}

\begin{proof}
We may assume that $1 < \frac{a}{p},  \frac{b}{p} <2$, and $\frac{a}p=[1, a_2, \ldots, a_k]$ with $a_2 \ge 2$.  Then we have  
$\frac{b}{p}=[1, 1, a_2-1, a_3, \ldots, a_k]$  
by Lemma~\ref{-r}. 
With the notation of the previous section, we have 
$$Q_{\frac{a}p}^{\cS} =Q(0,a_2-1, a_3, \ldots, a_k-1) 
\quad 
\text{and} 
\quad
Q_{\frac{b}p}^{\cS} = Q(a_2-1, a_3, \ldots, a_k-1) 
$$
(by  Remark~\ref{odd length}, we do not have to care about the parity of the length of the regular continued fraction).
Hence we have $Q_{\frac{b}p}^{\cS}=(Q_{\frac{a}p}^{\cS})^\vee$ by \eqref{(0,a)=a^v}, and $$\cS_{\frac{b}p}(q)=\cl(Q_{\frac{b}p}^{\cS})=\cl((Q_{\frac{a}p}^{\cS})^\vee)=\cl(Q_{\frac{a}p}^{\cS})^\vee=\cS_{\frac{a}p}^{\vee}(q)$$
by Theorem~\ref{theorem:MO closure} and \eqref{(0,a)=a^v}. 
\end{proof}

The following lemma is a variant of ``Palindrome Theorem"  (for example, see \cite[Theorem 4]{KL}) for continued fractions. We will give a direct proof here for the reader's convenience.  
 
\begin{lem}\label{1/r}
 Let $\frac{a}{p},  \frac{b}{p}$ be irreducible fractions with $\frac{a}p=[a_1, a_2, a_3, \ldots, a_n]$ as a regular continued fraction.  Set $b_1:= \left\lfloor \frac{b}{p} \right\rfloor$. Then $\frac{b}p=[b_1, a_{n}, a_{n-1}, \ldots,  a_2]$ if and only if $ab\equiv (-1)^n \pmod p$. 
 \end{lem}

\begin{proof}
Clearly, it suffices to show the case $a_1=b_1=0$. 
First, assume that $\frac{b}{p}=[0, a_n, \ldots, a_2]$. By Lemma~\ref{MO19 prop.3.1}, we have 
\[
\begin{pmatrix}
a&k\\p&l \end{pmatrix}
=\begin{pmatrix}
0&1\\1&0 \end{pmatrix}
\begin{pmatrix}
a_2&1\\1&0 \end{pmatrix}
\cdots
\begin{pmatrix}
a_n&1\\1&0 \end{pmatrix}
\]
for some  $k, l \in \ZZ$. 
Hence we have 
\[
\begin{pmatrix}
p&l\\a&k \end{pmatrix}
=\begin{pmatrix}
a_2&1\\1&0 \end{pmatrix}
\begin{pmatrix}
a_3&1\\1&0 \end{pmatrix}
\cdots
\begin{pmatrix}
a_n&1\\1&0 \end{pmatrix}.
\]
Taking the transpose of both sides, we get  
\[
\begin{pmatrix}
p&a\\l&k \end{pmatrix}
=\begin{pmatrix}
a_n&1\\1&0 \end{pmatrix}
\begin{pmatrix}
a_{n-1}&1\\1&0 \end{pmatrix}
\cdots
\begin{pmatrix}
a_2&1\\1&0 \end{pmatrix}.
\]
Hence we have 
\begin{equation}\label{inverse}
\begin{pmatrix}
l&k\\p&a \end{pmatrix}
=\begin{pmatrix}
0&1\\1&0 \end{pmatrix}
\begin{pmatrix}
a_{n}&1\\1&0 \end{pmatrix}
\cdots
\begin{pmatrix}
a_2&1\\1&0 \end{pmatrix},
\end{equation}
and it implies that $\frac{l}{p}=[0, a_n, \ldots, a_2]=\frac{b}p$, and hence $l=b$. The determinant of the right side of \eqref{inverse} is $(-1)^n$, so that of the left side is also. It implies that $ab-pk=(-1)^n$, and hence  
$ab \equiv (-1)^n \pmod{p}$.   

The converse implication follows from the above observation and the uniqueness of the solution of $\overline{a} \cdot x= \pm \overline{1}$ in $\ZZ/p\ZZ$.  
\end{proof}

\begin{prop}\label{dual2}
Let $\frac{a}{p},  \frac{b}{p}$ be irreducible fractions with $ab \equiv 1 \pmod{p}$. 
Then we have $\cS_{\frac{a}{p}}(q)=\cS^\vee_{\frac{b}{p}}(q)$. 
\end{prop}

\begin{proof} 
We may assume that $1 < \frac{a}{p},  \frac{b}{p} < 2$. 
If $\frac{a}p=[1, a_2, \ldots, a_{2m}]$, then $\frac{b}p=[1, a_{2m}, \ldots, a_2]$ by  Lemma~\ref{1/r}.  
Hence we have 
$$
Q_{\frac{b}p}^{\cS} = Q(a_{2m}-1, a_{2m-1}, \ldots,  a_2-1)^\vee \simeq Q(a_2-1, a_3, \ldots, a_{2m}-1 )=(Q_{\frac{a}p}^{\cS})^\vee$$
by Lemma~\ref{lem:closure of palindromic}. So the assertion follows from Theorem~\ref{theorem:MO closure} and \eqref{(0,a)=a^v}. 
\end{proof}

The following implies the sufficiency of Conjecture~\ref{1-7}.

\begin{thm}\label{sufficiency}
Let $p$ be a positive integer. 
For irreducible fractions $\frac{a}{p},  \frac{b}{p}$ with $ab \equiv -1 \pmod{p}$, we have 
$\cS_{\frac{a}{p}}(q)=\cS_{\frac{b}{p}}(q)$. 
\end{thm}

\begin{proof}
The assertion follows from Propositions~\ref{dual1}, \ref{dual2} and the fact that $f^{\vee\vee}(q)=f(q)$ for general $f(q)\in \ZZ[q]$. 
\end{proof}

{We note that, for the numerator $\mathcal{R}_{\frac{r}{s}}(q)$ for $\frac{r}s >1$, a similar result holds. See Lemma \ref{KRS-conj.2} below. }

Regarding $Q_\alpha^\cS$ as a finite poset, O\v{g}uz and Ravichandran \cite{OR} intensely studied $\cS_\alpha(q)$ from purely combinatorial point of view. 
Among other things, they showed that $\cS_\alpha(q)$ is always unimodal. 
Here we apply their another result. A polynomial $f(q)$ is said to be \textit{palindromic} if $f^{\vee}(q)=f(q)$. 
\begin{thm}\label{palindromic}
For an irreducible fraction $\frac{r}{s}$, $\cS_{\frac{r}{s}}(q)$ is palindromic if and only if $r^2 \equiv 1 \pmod{s}$.
\end{thm}

\begin{proof}
Let ${\mathbf b}=(b_1, \ldots, b_k)$ be an integer sequence  such that  $b_1, b_k \ge 0$, $b_2, \ldots, b_{k-1}>0$ and $k$ is odd. 
\cite[Theorem~1.3 (c)]{OR}, which was first conjectured in \cite{TBEC21},  states that $\cl(Q({\mathbf b}))$ is palindromic if and only if $b_i=b_{k+1-i}$ for all $1 \le i \le k$.

Set $\frac{r}s=[a_1, \ldots, a_{2m}]$. 
By the above mentioned result, the $q$-polynomial $\cS_{\frac{r}{s}}(q)=\cl(Q(a_2-1, a_3, \ldots, a_{2m}-1)^{\vee})$ is palindromic if and only if 
\begin{equation}\label{symmetric}
a_i =a_{2m+2-i} \quad  \text{for all $2 \le i \le 2m$.}
\end{equation}
By Lemma~\ref{1/r}, the condition \eqref{symmetric} holds if and only if $a^2 \equiv 1 \pmod{p}$. 
\end{proof}

\begin{cor}\label{power of p} The following hold.
\begin{enumerate}
\item[\textrm{(1)}] For an irreducible fraction $\frac{a}{p^n}$ such that $p$ is an odd prime, $\cS_{\frac{a}{p^{n}}}(q)$ is palindromic, if and only if $a \equiv \pm 1 \pmod{p^n}$, if and only if  $\cS_{\frac{a}{p^n}}(q)=[p^n]_q=1+q+\cdots +q^{p^{n}-1}$.  

\item[\textrm{(2)}] For $n \ge 2$, $\cS_{\frac{a}{2^n}}(q)$ is palindromic if and only if $a \equiv \pm 1 \pmod{2^n}$ or $a \equiv 2^{n-1} \pm 1 \pmod{2^n}$. 
\end{enumerate}
\end{cor}

\begin{proof}
(1) The latter equivalence is clear, so we prove the former. 
By Theorem~\ref{palindromic}, it is sufficient to show that $a^2 \equiv 1 \pmod{p^n}$ implies $a \equiv \pm 1 \pmod{p^n}$.  If $a^2 \equiv 1 \pmod{p^n}$, then $p^n$ divides $(a+1)(a-1)$. Since $p$ is an odd prime, $p$ does only divide one of $a+1$ and $a-1$. In fact, if $p$ divides both $a+1$ and $a-1$, then $p$ divides $2$, which is a contradiction. This means that all $n$ copies of $p$ that appear in the prime decomposition of $(a+1)(a-1)$ must come from either $a+1$ or $a-1$. Thus, $p^n$ divides either $a+1$ or $a-1$, equivalently, $a \equiv \pm 1 \pmod{p^n}$. 

(2) Since $4$ cannot divide both $a+1$ and $a-1$ at the same time, an argument similar to the above works.
\end{proof}

Combining the above results with Chinese remainder theorem, for a general  $s$, we can easily detect all $r$ such that $\cS_{\frac{r}s}$ is palindromic (equivalently, $r^2 \equiv 1 \pmod{s}$).

\begin{cor}\label{R: palindromic}
For an irreducible fraction $\frac{r}s >1$, $\cR_{\frac{r}s}(q)$ is palindromic if and only if $s^2 \equiv 1 \pmod{r}$. 
\end{cor}

\begin{proof}
By \eqref{wakui-1.6-2}, we have $\cR_{\frac{r}s}(q)=\cS_{\frac{r-s}r}(q)$. Hence we have 
\begin{eqnarray*}
\text{$\cR_{\frac{r}s}(q)$ is palindromic} &\Longleftrightarrow& \text{$\cS_{\frac{r-s}r}(q)$ is palindromic}\\
&\Longleftrightarrow&  (r-s)^2 \equiv 1 \pmod{r} \\
&\Longleftrightarrow& s^2 \equiv 1 \pmod{r},
\end{eqnarray*}
where the second equivalence follows from Corollary~\ref{palindromic}. 
\end{proof}

\section{Another proof of the sufficiency of the conjecture} \label{sec-4}

\par 
In \cite{Wakui_LDTandNT}, the fourth author introduced the $q$-deformed integers derived from pairs of positive and coprime integers. 
In this section, by using them we give the second proof of the sufficiency of Conjecture~\ref{1-7}. 
To do this we need the following interpretation of the conjecture. 

\begin{lem}\label{KRS-conj.2}
Conjecture~\ref{1-7} is equivalent to the following statement. 
\begin{itemize}
\item[$(*)$] Let $p$ be an odd prime integer. 
For two integers $a, b$ with $1\leq a<b<p$, 
$\mathcal{R}_{\frac{p}{a}}(q)=\mathcal{R}_{\frac{p}{b}}(q)$ if and only if $ab\equiv -1\pmod p$. 
\end{itemize}
\end{lem}
\begin{proof}
This follows from Lemmas \ref{wakui-1.5} and \ref{wakui-1.6}. 
\end{proof}

\begin{dfn}[{\cite[Definition 4.3]{Wakui_LDTandNT}}]\label{2-1}
For a pair $(a, b)$ of positive and coprime integers we define a polynomial $(a, b)_q$ in $q$ with integer coefficients by 
\begin{equation}\label{eq2-1}
(a, b)_q
:=\begin{cases}
(a-r, r)_q+q(a, b-a)_q & \text{if $a<b$},\\ 
(a-b, b)_q+q^{\lceil \frac{a}{b}\rceil }(r, b-r)_q & \text{if $a>b$}, 
\end{cases}
\end{equation}
where $r$ is the remainder when $b$ is divided by $a$ in case where $a<b$, and 
when $a$ is divided by $b$ in case where $a>b$, 
and also $(1, n)_q=(n, 1)_q:=[1+n]_q$ for any non-negative integer $n$. 
\end{dfn}

The polynomial $(a, b)_q$ is convenient to compute $[\alpha \sharp \beta ]_q$. 

\begin{thm}[{\cite[Theorem 4.4]{Wakui_LDTandNT}}]\label{2-3}
If $\alpha =\frac{x}{a},\ \beta =\frac{y}{b}\geq 1$ are Farey neighbors, then 
$$
\mathcal{S}_{\alpha \sharp \beta }(q)=(a, b)_q,\quad \mathcal{R}_{\alpha \sharp \beta }(q)=(x, y)_q. $$
Thus,  we have 
$$[\alpha \sharp \beta ]_q=\dfrac{(x,y)_q}{(a, b)_q}.$$ 
\end{thm}

Any rational number $\alpha >0$ is associated with a link $L(\alpha )$ in the $3$-sphere $\mathbb{S}^3$ which is given by the diagram $D(\alpha )$ below, and such a link is called a \textit{rational link} or \textit{two-bridge link}. 
If $\alpha $ belongs to the open interval $(0,1)$, then the diagram $D(\alpha )$ is given as in Figure~\ref{fig1} after the expression of $\alpha =[0, a_1, \ldots , a_n]$ with odd $n$. 

\begin{figure}[ht]
    \centering
    \includegraphics[height=2cm, width=13cm]{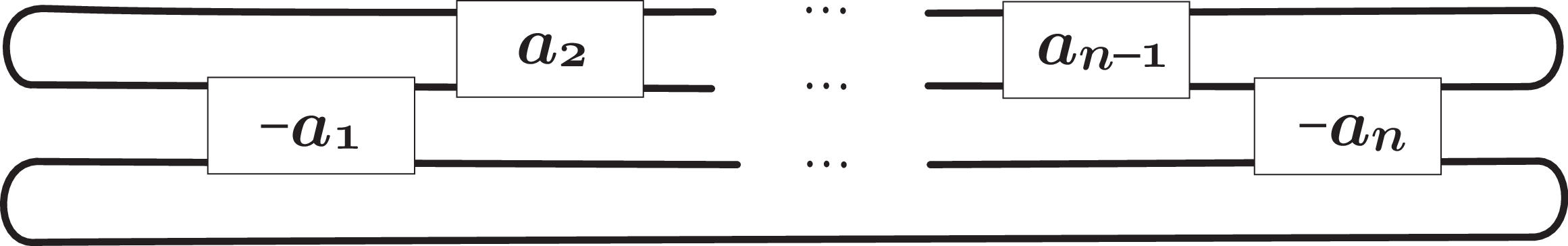}
    \caption{the diagram $D( \alpha )$ of rational link}\label{fig1}
\end{figure}
\noindent where
$$\raisebox{-0.2cm}{\includegraphics[height=0.6cm]{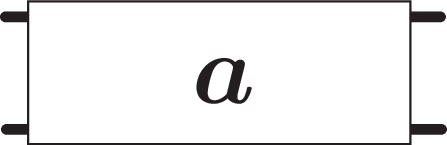}}\ \ =\ 
\begin{cases}
 \raisebox{-0.5cm}{\includegraphics[height=1.6cm]{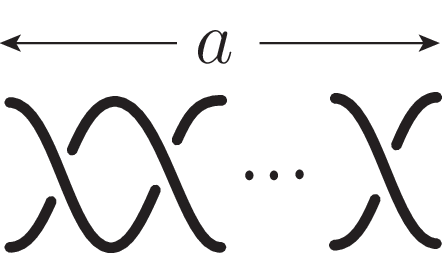}}\ & \raisebox{-0.2cm}{if $a\geq 0$,}\\[0.5cm]  
\raisebox{-0.5cm}{\includegraphics[height=1.6cm]{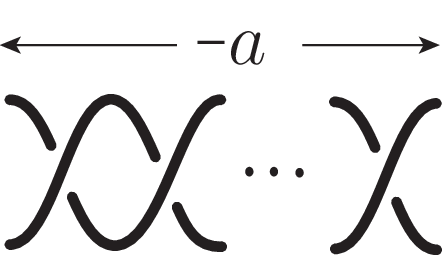}}\ & \raisebox{-0.2cm}{if $a< 0$.}
\end{cases} $$
If $\alpha >1$, then $D(\alpha )$ is defined by $D(\alpha ):=D(\alpha ^{-1})$, and 
if $\alpha =1$, then $D(\alpha )=~\raisebox{-0.3cm}{\includegraphics[height=0.8cm]{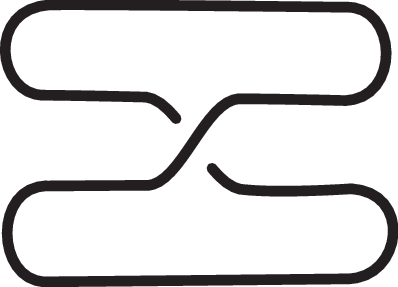}}$.  
For a negative rational number $\alpha$, a rational link $L(\alpha )$ and its diagram $D(\alpha )$ are defined in the same way as the positive case. 
Then we see that the link $L(\alpha )$ is the mirror image of $L(-\alpha )$. 
However, for any $\alpha \in \QQ$, there is some $\beta \in \QQ \cap(1, \infty)$ such that $L(\alpha)$ and $L(\beta)$ are isotopy. See, for example, \cite[Theorem~2]{KL}. In this sense, we may assume that $\alpha >1$. 

\par 
As a useful isotopy invariant for an oriented link $L$ in $\mathbb{S}^3$, the Jones polynomial $V_{L}(t)$ \cite{Jones, Kau-Topology}, which is valued in $\mathbb{Z}[t^{\pm \frac{1}{2}}]$, is well-studied. 
Lee and Schiffler~\cite{LeeSchiffler} introduced the following normalization $J_{\alpha }(q)$ of the Jones polynomial $V_{\alpha}(t):=V_{L(\alpha )}(t)$ of a rational link $L(\alpha )$: 
\begin{equation}\label{def normalized Jones}
 J_{\alpha }(q):=\pm t^{-h}V_{\alpha }(t)|_{t=-q^{-1}},
\end{equation}
where $\pm t^{h}$ is the leading term of $V_{\alpha }(t)$. 
This indicates the normalization such that the constant term is $1$ as a polynomial in $q$. 
We note that 
\[ J_1(q)=1,\quad  J_{\infty}(q)=q. \]

By Lee and Schiffler~\cite{LeeSchiffler}, it is known that the Jones polynomial $V_{\alpha }(t)$ can be recovered from $J_{\alpha }(q)$. 
By \cite[Proposition A.1]{MO} and the equation \eqref{q-deformed farey sum}, we see that, for a rational number $\alpha >1$,  the normalized Jones polynomial $J_{\alpha}(q)$ can be computed by 
\begin{equation}\label{eq2-6}
J_{\alpha }(q)=q\mathcal{R}_{\alpha }(q)+(1-q)\mathcal{S}_{\alpha }(q).  
\end{equation} 
Using this formula, the fourth author showed the following.

\begin{thm}[{\cite[Theorem 5.3]{Wakui_LDTandNT}}]\label{2-5}
Let $(a, p)$ be a pair of coprime integers with $1\leq a<p$. Then 
\begin{equation}\label{eq2-8}
(a, p)_q=J_{\frac{p}{a}}(q)+q(a-r, r)_q, 
\end{equation}
where $r$ is the remainder when $p$ is divided by $a$. 
\end{thm}

The equation \eqref{eq2-8} corresponds to the equation 
$J_{\frac{r}{s}}(q)=\mathcal{R}^{\prime}-q\mathcal{S}^{\prime}$ in {\cite[p.45]{MO}} under the setting $\frac{r}{s}=\frac{p}{a}$. 
{In fact, since 
$\mathcal{S}^{\prime}=\mathcal{S}_{\frac{p}{a}}(q)$ and 
$\mathcal{R}^{\prime}=q\mathcal{R}_{\frac{p}{a}}(q)+\mathcal{S}_{\frac{p}{a}}(q)$ as shown in {\cite[p.45]{MO}},}
we have 
$\mathcal{S}^{\prime}=(a-r, r)_q$ by \eqref{eq2-10} below and 
$\mathcal{R}^{\prime}=q(a, p-a)_q+(a-r, r)_q=(a, p)_q$ by \eqref{eq2-1} and \eqref{eq2-9} below. 
\par

As an application of Theorem~\ref{2-5} we have: 

\begin{thm}[{\cite[Theorem 5.4]{Wakui_LDTandNT}}]\label{2-6}
Let $(a, p)$ be a pair of coprime integers with $1\leq a\leq p$, and $r$  the remainder when $p$ is divided by $a$. Then, the following equations hold.
\begin{align}
\mathcal{R}_{\frac{p}{a}}(q) &=(a, p-a)_q,  \label{eq2-9}\\ 
\mathcal{S}_{\frac{p}{a}}(q)  &=(a-r, r)_q. \label{eq2-10}
\end{align}
\end{thm}

By using the above theorem, one can give another proof of Theorem~\ref{sufficiency}, that is, the sufficiency of Conjecture~\ref{1-7}.

\begin{proof}[Another proof of Theorem~\ref{sufficiency}.]
 Since $ab \equiv -1 \pmod p$, there is some $m\in\ZZ$ with $mp-ab=1$. Since $1\leq a\leq b<p$, we have $0\leq m<a$. Thus, $\left(\frac{a}{m},\frac{p-a}{b-m}\right)$ is the Farey parents of $\frac{p}{b}$. By using Theorem~\ref{2-3}, we have 
 $$\mathcal{R}_{\frac{p}{b}}(q)=(a, p-a)_q.$$
(See also the proof of \cite[Theorem 3.2]{Kogiso-Wakui2}.) 
Combining this equation with \eqref{eq2-9}, we get $\mathcal{R}_{\frac{p}{a}}(q)=\mathcal{R}_{\frac{p}{b}}(q)$. 
\end{proof} 

\begin{rem}
By Corollary~\ref{R: palindromic} and \eqref{eq2-9}, $(a,b)_q$ is palindromic if and only if $a^2 \equiv 1 \pmod{a+b}$. Combining this observation with Theorem~\ref{2-3}, one can show the following.  For a positive irreducible fractions $\frac{c}z$ whose Farey parent is $(\frac{a}x, \frac{b}y)$, $\cS_{\frac{c}{z}}(q)$ is palindromic, 
if and only of $x^2 \equiv 1 \pmod{x+y}$ (equivalently, $y^2 \equiv 1 \pmod{x+y}$).      
\end{rem}

\section{Three operations on the positive rational numbers and
\texorpdfstring{$q$}{q}-deformed rational numbers} \label{sec-5}

{In the study of Conway-Coxeter friezes of zigzag-type developed by the first and the fourth authors \cite{Kogiso-Wakui_Proc, Kogiso-Wakui}
crucial three operators $\mathfrak{i}, \mathfrak{r}, \mathfrak{ir}$ on the positive rational numbers are introduced. 
In this section we examine effect of the operators $\mathfrak{i}, \mathfrak{r}, \mathfrak{ir}$ on $\mathcal{R}_{\alpha}(q), \mathcal{S}_{\alpha}(q)$.}

Let $\alpha =\frac{z}{c} >0$ be an irreducible fraction.
In the case where $\alpha \in (0, 1)$, irreducible fractions $\mathfrak{i}(\alpha ), \mathfrak{r}(\alpha ), (\mathfrak{ir})(\alpha )$ in the interval $(0,1)$ are defined as follows \cite{Kogiso-Wakui}:
\begin{equation}\label{eq3-1}
\mathfrak{i}(\alpha ):=\dfrac{c-z}{c}\ (=1-\alpha ),\quad 
\mathfrak{r}(\alpha ):=\dfrac{a}{c},\quad 
\mathfrak{ir}(\alpha ):=\dfrac{b}{c},  
\end{equation} 
where $(\frac{x}{a}, \frac{y}{b})$ is the Farey parent of $\alpha$. 
Note that
\begin{equation}\label{r and i}
az \equiv 1 \pmod{c} \quad \text{and} \quad  b\equiv-a \pmod{c}.
\end{equation}
In fact, since $z=x+y$, $c=a+b$ and $ay-bx=1$ now, we have 
$$az=a(x+y)=ax+ay=ax+1+bx=1+(a+b)x=1+cx.$$
Hence as operations on $\QQ \cap (0,1)$, we have $\mathfrak{i}^2=\mathfrak{r}^2=\operatorname{id}$ and $\mathfrak{ir}=\mathfrak{ri}$. By Theorem~\ref{palindromic}, for $\alpha \in \QQ \cap (0,1)$, $\mathfrak{r}(\alpha)=\alpha$ if and only if $\cS_\alpha(q)$ is palindromic.

In the case where $\alpha >1$, $\mathfrak{i}(\alpha ), \mathfrak{r}(\alpha )$, and $\mathfrak{ir}(\alpha )$ are defined as follows: 
\begin{equation}\label{eq3-3}
\mathfrak{i}(\alpha):=\left( \mathfrak{i}(\alpha^{-1} )\right) ^{-1},\qquad 
\mathfrak{r}(\alpha):=\left( \mathfrak{r}(\alpha^{-1} )\right) ^{-1},\qquad 
(\mathfrak{ir})(\alpha):=\left( (\mathfrak{ir})(\alpha^{-1})\right) ^{-1}. 
\end{equation}
Here we also have $\mathfrak{i}^2=\mathfrak{r}^2=\operatorname{id}$ and $\mathfrak{ir}=\mathfrak{ri}$. Moreover, $\mathfrak{r}(\alpha)=\alpha$ if and only if $\cR_\alpha(q)$ is palindromic. 

By Lemmas \ref{-r} and \ref{1/r}, and the equation \eqref{r and i}, one can show that for a positive rational number $\alpha =[0, a_2, \ldots , a_n]$, 
\begin{align}
(\mathfrak{ir})(\alpha )&={\begin{cases}
[0, a_n, \ldots , a_3, a_2] & \text{if $n$ is odd},  \\ 
[0, 1, a_n-1, a_{n-1}, \ldots , a_3, a_2] & \text{if $n$ is even and $a_n \ge 2$},  \\
[0, a_{n-1}+1, a_{n-2}, \ldots , a_3, a_2] & \text{if $n$ is even and $a_n =1$,}
\end{cases}} \label{eq3-4}\\[0.1cm]  
\mathfrak{i}(\alpha )&={\begin{cases}[0, 1, a_2-1, a_3, \ldots , a_n] & \text{if $a_2 \geq 2$}, \\
[0, a_3+1, a_4, \ldots , a_n] & \text{if $a_2=1$},  
\end{cases}}
\label{eq3-5}\\[0.1cm]  
\mathfrak{r}(\alpha )&={\begin{cases}
[0, 1, a_n-1, a_{n-1}, \ldots , a_3, a_2] & \text{if $n$ is odd and $a_n \ge 2$},\\ 
[0, a_{n-1}+1, a_{n-2}, \ldots , a_3, a_2] & \text{if $n$ is odd and $a_n =1$,}\\
[0, a_n, \ldots , a_3, a_2]& \text{if $n$ is even}. 
\end{cases}} \label{eq3-6}
\end{align}

\begin{lem}\label{farey sums and continued fractions}
Let $\alpha \in \mathbb{Q}\cap (0, \infty )$ whose expression as a regular continued fraction  is $\alpha =[ a_1, a_2, \ldots , a_n]$. 
If $n$ is odd, then 
\begin{align*}
\beta &={\begin{cases}
 [ a_1, \ldots , a_{n-1}, a_n-1] & \text{if $a_n\geq 2$},\\ 
 [ a_1, \ldots , a_{n-2}] & \text{if $a_n=1$ and $n\geq3$},\\
[0 ] &\text{if $a_1=1$ and  $n=1$,}
\end{cases}}\\ 
\gamma &={\begin{cases}
[ a_1, \ldots , a_{n-1}] & \text{if $n\geq3$,}\\
[\quad ] & \text{if $n=1$}
\end{cases}}
\end{align*}
are Farey neighbors and $\alpha =\beta \sharp \gamma$, where $[\quad]$ expresses $\infty=\frac{1}0$.
If $n$ is even, then 
\begin{align*}
\beta &=[ a_1, \ldots , a_{n-1}] \\
\gamma &={\begin{cases}
[ a_1, \ldots , a_{n-1}, a_n-1] & \text{if $a_n\geq 2$},\\ 
[ a_1, \ldots , a_{n-2}] & \text{if $a_n=1$ and $n\geq4$,}\\
[\quad ] & \text{if $a_2=1$ and $n=2$}
\end{cases}}
\end{align*}
are Farey neighbors and $\alpha =\beta \sharp \gamma $. 
\end{lem}

\begin{proof} 
Note that, as a regular continued fraction, each rational number has expressions in both odd and even lengths. It is easy to check that two definitions (the odd case and the even case) coincide. 

We will prove the equation $\alpha =\beta \sharp \gamma$ by induction on $n$. The cases $n=0$ and $n=1$ are clear. 
Now, we suppose that the statement holds for $n-1$.
We only show the case $n \ge 3$ is odd and $a_n\geq 2$; the proofs in other cases are similar. 
Not that we have  $\beta=[a_1,\ldots,a_{n-1},a_n-1]$ and  $\gamma=[a_1,\ldots,a_{n-1}]$ now. 
Set $[a_2,\ldots,a_{n-1}]=\frac{r}{s}$, $[a_2,\ldots,a_{n-1},a_n-1]=\frac{r'}{s'}$.  
 
By  induction hypothesis, $(\frac{r}s, \frac{r'}{s'})$ is the Farey parents of $\frac{r+r'}{s+s'}=[a_2, a_3, \ldots, a_n]$. 
Since $\beta= \frac{r'a_1+s'}{r'}$, $\gamma= \frac{ra_1+ s}{r}$ and  $sr'-rs'=1$, $\beta$ and $\gamma$ are Farey neighbors.
Moreover, it follows from induction hypothesis that we have 
$$
\beta \sharp \gamma
=\dfrac{a_1(r+r')+s+s'}{r+r'}=a_1+\dfrac{1}{\ \dfrac{r+r'}{s+s'}\ } \\
=a_1+\dfrac{1}{[a_2,\ldots,a_n]}=[a_1, a_2, \ldots, a_n]=\alpha.$$
\end{proof}

For a quiver $Q$ of type $A$, 
let denote by $Q^{\mathsf{rot}}$ the quiver obtained from $Q$ by $\pi$-rotation. Since $Q\simeq Q^{\mathsf{rot}}$ as quivers, their closure polynomials are same;
\begin{equation}\label{rotation closure poly.}
\cl(Q)=\cl(Q^\mathsf{rot}).    
\end{equation}

For $\alpha \in \mathbb{Q}\cap (1, \infty )$, the equations \eqref{eq3-4}, \eqref{eq3-5} and  \eqref{eq3-6} imply that 
\begin{align}
Q_{\mathfrak{\mathfrak{i}}(\alpha )}^\mathcal{R}&=(Q_{\alpha }^\mathcal{R})^{\vee}, \label{eq3-13}\\ 
Q_{(\mathfrak{ir})(\alpha )}^\mathcal{R}&=(Q_{\alpha }^\mathcal{R})^{\mathsf{rot}},  \label{eq3-14}\\ 
Q_{\mathfrak{r}(\alpha )}^\mathcal{R}&=(Q_{\alpha }^\mathcal{R})^{\mathsf{rot}\vee}=(Q_{\alpha }^\mathcal{R}) ^{\vee\mathsf{rot}}.  \label{eq3-15}
\end{align}

Except for the denominator of $\mathfrak{i}(\alpha )$, the denominator and numerator polynomials of $q$-deformations of $\mathfrak{i}(\alpha ), \mathfrak{r}(\alpha ), (\mathfrak{ir})(\alpha )$ are computed from that of $\alpha $ and its Farey parent as follows. 

\begin{thm}\label{3-6}
Let $\alpha \in \mathbb{Q}\cap (1, \infty )$ and $(\beta , \gamma )$ be its parents. 
Then, the following hold.
\begin{enumerate}
\item[(1)] $\mathcal{R}_{\mathfrak{i}(\alpha )}(q) =
\mathcal{R}_{ \mathfrak{r}(\alpha )}(q) =\mathcal{R}_{\alpha }^{\vee}(q)$ and 
$\mathcal{R}_{ (\mathfrak{ir})(\alpha )}(q) =
\mathcal{R}_{\alpha }(q)$. 
\vspace{0.1cm} 
\item[(2)] $\mathcal{S}_{(\mathfrak{ir})(\alpha )}(q)=\mathcal{R}_{\beta  }(q)$, and 
$\mathcal{S}_{\mathfrak{r}(\alpha )}(q) =\mathcal{R}_{\gamma }^{\vee}(q) $. 
\end{enumerate}
\end{thm}
\begin{proof}
(1) By the equations \eqref{opposite closure poly 2} and \eqref{rotation closure poly.}, these follow from \eqref{eq3-13}, \eqref{eq3-14}, and \eqref{eq3-15}.

(2) We write $\alpha $ as $\alpha =[a_1, a_2, \ldots , a_{2m}]$.
Then, it follows from Lemma \ref{farey sums and continued fractions} that $\beta=[a_1, \ldots , a_{2m-1}]$.
Since $\alpha^{-1}=[0,a_{1},\ldots, a_{2m}]\in\mathbb{Q}\cap(0,1)$, we have $(\mathfrak{ir})(\alpha^{-1} )=[0, a_{2m},  \ldots, a_1]$ by \eqref{eq3-4}. Thus, 
$(\mathfrak{ir})(\alpha )=[a_{2m},  \ldots , a_1]$. 
Hence, Theorem \ref{theorem:MO closure} and \eqref{rotation closure poly.} imply that
\begin{align*} 
\mathcal{S}_{(\mathfrak{ir})(\alpha )}(q) 
& =\cl(Q(0, a_{2m-1}-1,a_{2m-2},\ldots, a_{2},a_1-1)) \\
& =\cl(Q(0, a_{2m-1}-1,a_{2m-2},\ldots, a_{2},a_1-1)^\mathsf{rot}) \\
& =\cl(Q(a_{1}-1,a_{2},\ldots, a_{2m-2},a_{2m-1}-1))\\
& =\mathcal{R}_{\beta}(q). 
\end{align*}

Finally, we consider $\mathcal{S}_{\mathfrak{r}(\alpha )}(q)$.
Suppose that $a_{2m}>1$.
In this case, it follows from Lemma \ref{farey sums and continued fractions} that $\gamma=[a_1, \ldots , a_{2m-1},a_{2m}-1]$.
By \eqref{eq3-6}, $\mathfrak{r}(\alpha)=[1,a_{2m}-1,a_{2m-1},\ldots, a_1]$.
Thus, we have
\begin{align*} 
\mathcal{S}_{\mathfrak{r}(\alpha )}(q) 
& =\cl(Q(a_{2m}-2,a_{2m-1},\ldots, a_{2},a_1-1)^\vee) \\
&  =\cl(Q(a_{2m}-2,a_{2m-1},\ldots, a_{2},a_1-1))^\vee \\
&  =\cl(Q(a_{2m}-2,a_{2m-1},\ldots, a_{2},a_1-1)^{\mathsf{rot}})^\vee \\
& =\cl(Q(a_{1}-1,a_{2},\ldots, a_{2m-1},a_{2m}-2))^\vee\\
& =\mathcal{R}_{\gamma }^\vee(q). 
\end{align*}
In the case where $a_{2m}=1$, by the same argument the same equation is derived.
\end{proof}

For a rational number $\alpha $ with $0<\alpha <1$, the $q$-deformations of $\mathfrak{i}(\alpha ), \mathfrak{r}(\alpha )$ and $(\mathfrak{ir})(\alpha )$ behave as follows. 

\begin{prop}\label{3-7}
For a rational number $\alpha \in \mathbb{Q}\cap (0,1)$, we have the followings.
\begin{enumerate}
\item[(1)] $\mathcal{S}_{\mathfrak{r}(\alpha )}(q) = \mathcal{S}_{\alpha }^{\vee}(q)=\mathcal{S}_{\mathfrak{i}(\alpha )}(q) =\mathcal{R}_{\alpha ^{-1}}(q)$. 
\item[(2)] $\mathcal{R}_{\mathfrak{i}(\alpha )}(q) =\mathcal{R}_{\alpha ^{-1}}(q) -\mathcal{S}_{\alpha ^{-1}}(q),\ 
\mathcal{R}_{r(\alpha )}(q) =\mathcal{R}_{\alpha ^{-1}}(q) -\mathcal{R}_{\gamma ^{-1}}(q)$, 
where $(\beta , \gamma )$ is the parent of $\alpha $. 
\end{enumerate}
\end{prop}
\begin{proof}
(1) The first (resp. second) equality follows from \eqref{eq3-1}, \eqref{r and i}, and Proposition~\ref{dual2} (resp. Proposition~\ref{dual1}). To see the third equality, express $\mathfrak{i}(\alpha)=\frac{a}x$ as an irreducible fraction. Then we have $\alpha=\frac{x-a}x$ and $\alpha^{-1}=\frac{x}{x-a}$. So the equality follows from \eqref{wakui-1.6-2}.

(2) The first equation immediately follows from \eqref{wakui-1.6-1}. Since $\mathfrak{i}(\mathfrak{ir}(\alpha))=\mathfrak{r}(\alpha)$ and $(\mathfrak{ir}(\alpha))^{-1}=\mathfrak{ir}(\alpha^{-1})$, replacing $\alpha$ by $\mathfrak{ir}(\alpha)$, the first equation yields 
\[
\mathcal{R}_{\mathfrak{r}(\alpha )}(q)
=\mathcal{R}_{(\mathfrak{ir})(\alpha ^{-1})}(q) -\mathcal{S}_{(\mathfrak{ir})(\alpha ^{-1})}(q).
\]
Applying Theorem~\ref{3-6}, we have 
$\mathcal{R}_{ \mathfrak{r}(\alpha )}(q)
=\mathcal{R}_{ \alpha ^{-1}}(q) -\mathcal{R}_{\gamma ^{-1}}(q)$. 
\end{proof} 

As an application of Proposition \ref{dual1} we have: 

\begin{thm}\label{3-9}
Let $\alpha \in \mathbb{Q}\cap (1, \infty )$, and express it as a regular continued fraction $\alpha =[a_1, a_2, \ldots , a_n]$. 
\begin{enumerate}
\item[(1)] If $a_1=1$, then 
\[
\mathcal{S}_{\mathfrak{i}(\alpha )}(q) =\mathcal{S}_{\frac{(a_2+1)(\alpha -1)+\alpha -2}{\alpha -1}}^{\vee}(q).
\]
\item[(2)] If $a_1\geq 2$, then 
\[
\mathcal{S}_{\mathfrak{i}(\alpha )}(q) =\mathcal{S}_{\frac{a_1(\alpha -1)+\alpha -2}{\alpha -1}}^{\vee}(q).
\]
\end{enumerate}
\end{thm}
\begin{proof}
Set $\alpha =\frac{x}{a}$ with $1\leq a<x$. Then, 
$\mathcal{S}_{\mathfrak{i}(\alpha )}(q)
=\mathcal{S}_{\frac{x}{x-a}}(q)$, and hence by Proposition~\ref{dual1}
\begin{equation}\label{natural}
\mathcal{S}_{\frac{x}{x-a}}(q) =\mathcal{S}_{\frac{x^{\prime}}{x-a}}^{\vee}(q) 
\end{equation}
for $x^{\prime}\in \mathbb{Z}$ such that $x^{\prime}\equiv -x \pmod{x-a}$ and $\bigl\lfloor \frac{x}{x-a}\bigr\rfloor=\bigl\lfloor \frac{x^{\prime}}{x-a} \bigr\rfloor$. 

(1) Since $a_1=1$, we have 
$\frac{a}{x-a}=[a_2, \ldots , a_n]$, and $0\leq a-a_2(x-a)<x-a$. 
As $x^{\prime}$ one can take $x^{\prime}:=-x+(a_2+3)(x-a)$.  
Thus by \eqref{natural} we have 
\begin{align*}
\mathcal{S}_{\frac{x}{x-a}}(q)
&=\mathcal{S}_{\frac{-x+(a_2+3)(x-a)}{x-a}}^{\vee}(q)  \\ 
&=\mathcal{S}_{\frac{-\alpha +(a_2+3)(\alpha -1)}{\alpha -1}}^{\vee}(q)\\ 
&=\mathcal{S}_{\frac{(a_2+1)(\alpha -1)+\alpha -2}{\alpha -1}}^{\vee}(q). 
\end{align*}

(2) Since $a_1\geq 2$, we have 
$\frac{x}{a}-a_1=[0, a_2, \ldots , a_n]$ and $0\leq x-aa_1<a$. 
In this case 
one can take $x^{\prime}:=-x+(a_1+2)(x-a)$. 
Then, by the same argument of the proof of Part (1), the assertion is derived. 
\end{proof} 

\section{A formula for computing closure polynomials of type \texorpdfstring{$A$}{A}}\label{sec-6}

{On the one hand, for an irreducible fraction $\alpha >1$, the denominator and numerator polynomials of $[\alpha]_q$ are given by closure polynomials of some quivers of type $A$ (see  Theorem \ref{theorem:MO closure}).
On the other hand, from a representation theoretical viewpoint, the closure polynomial of a type $A$ quiver $Q$ counts subrepresentations of  ``the full interval representation" of $Q$ in which a field $\mathsf{k}$ corresponds to each vertex and the identity map corresponds to each arrow.
In this section, we give an expression to calculate $\cl(Q)$ that explicitly gives the number of subrepresentations of the full interval representation.}\\

{Let $Q$ be a quiver of type $A$, that is, the underlying graph of $Q$ is $A_n=1-2-3-\cdots-n$.}
A \textit{representation} of $Q$ over a field $\mathsf{k}$ is a system $M=(M_{a},\varphi_{\alpha})_{a\in Q_0,\alpha\in Q_1}$ ($M=(M_a,\varphi_\alpha)$ for short) consisting of $\mathsf{k}$-vector spaces $M_a$ $(a\in Q_0)$, and $\mathsf{k}$-linear maps $\varphi_\alpha:M_{s(\alpha)}\to M_{t(\alpha)}$ ($\alpha\in Q_1$). 
{The \textit{dimension} of $M$ is the sum of $\mathsf{k}$-dimensions of $M_{a}$. }
A representation $M'=(M'_a,\varphi'_\alpha)$ is said to be a \textit{subrepresentation} of $M$ if $M'_a$ is a subspace of $M_a$, and $\varphi_\alpha'=\varphi_\alpha|_{M'_a}$.
For two representations $M=(M_a,\varphi_\alpha)$ and $N=(N_a,\psi_\alpha)$, a \textit{morphism} of representations $f:M\to N$ is a family $f=(f_a)_{a\in Q_0}$ of $\mathsf{k}$-linear maps $f_a:M_a\to N_a$ such that $\psi_\alpha f_{s(\alpha)}=f_{t(\alpha)}\varphi_\alpha$ for any arrow $\alpha$.

The category of finite dimensional representations of $Q$ is denoted by $\mathsf{rep}(Q)$.
It is well-known that there is an $\mathsf{k}$-linear equivalence between $\mathsf{rep}(Q)$ and the category of finitely generated $\mathsf{k}Q$-modules, where $\mathsf{k}Q$ is the path algebra of $Q$.
For a vertex $i \in Q_0$, we denote by $S(i)$ the corresponding simple $\mathsf{k}Q$-module.
{For a $\mathsf{k}Q$-module $M$, we also denote by $\mathsf{rad}(M)$, $\mathsf{top}(M)$, and $\mathsf{soc}(M)$ the Jacobson radical, the top, and the socle of $M$, respectively.}
The \textit{support} of $M$ is the set of composition factors, which is denoted by $\mathsf{supp}(M)$.
In this subsection, any objects of $\mathsf{rep}(Q)$ are freely regarded as objects of $\mathsf{mod}~\mathsf{k}Q$.
For representations of quivers, see \cite[Chapters II and III]{ASS} for more details.

By the Gabriel theorem (for example, see \cite[Chapter VII, Theorem 5.10]{ASS}), there is one-to-one corresponding between indecomposable objects of $\mathsf{rep}(Q)$ and positive roots of $A_n$, that is, pairs $(i,j)$ with $1\leq i\leq j \leq n$.
In this correspondence, each pair $(i,j)$ is assigned with the interval representation $\mathbb{I}[i,j]=(M_a,\varphi_\alpha)$, where
\[
M_a=\left\{\begin{array}{ll}
\mathsf{k} & \text{if $i\leq a\leq j$,} \\
\{0\} & \text{otherwise,}
\end{array}\right.
\quad
\varphi_\alpha=\left\{\begin{array}{ll}
1 & \text{if $i\leq s(\alpha)$ and $t(\alpha)\leq j$,} \\
0 & \text{otherwise.}
\end{array}\right.
\]
{Following this notation, $\mathbb{I}[1,n]$ is called the \textit{full interval representation} of $Q$.} 
Then, it follows from the definition of the closure polynomial that 
the coefficient of $q^{\ell}$ of $\mathsf{cl}(Q)$ is equal to the number of $\ell$-dimensional subrepresentations of $\mathbb{I}[1,n]$.\\

{Throughout this section, 
we fix a type $A$ quiver $Q=Q(\mathbf{a})$ for some tuple $\mathbf{a}=(a_1,a_2,\ldots, a_s)\in \mathbb{Z}^s_{\geq 0}$, which has $n$ vertices, and denote by $\mathbb{I}(\mathbf{a})$ the full representation of $Q(\mathbf{a})$. }
It is clear that $\rho_1(Q(\mathbf{a}))=\mathsf{dim}_\mathsf{k}~\mathsf{soc}({\mathbb{I}(\mathbf{a})})$, which equals to the number of sinks of $Q(\mathbf{a})$.
Note that the Jordan-H\"{o}lder theorem implies that any coefficients of $\cl(Q(\mathbf{a}))$ are greater than or equal to $1$.
This yields that, for any irreducible fraction $\alpha>1$, any coefficients of the polynomials $\cR_\alpha{(q)}$ and $\cS_\alpha{(q)}$ are greater than $1$. 
We also remark that the top and the socle of ${\mathbb{I}(\mathbf{a})}$ is given by
\begin{align*} 
& \mathsf{top}({\mathbb{I}(\mathbf{a})})=\bigoplus_{k\geq 1}S(1+a_1+a_2+\cdots +a_{2k-1}) \\
& \mathsf{soc}({\mathbb{I}(\mathbf{a})}))=\left\{\begin{array}{ll}
     S(1)\oplus \displaystyle\bigoplus_{k\geq 1}S(1+a_1+a_2+\cdots +a_{2k}) & \text{if $a_1\neq 0$,} \\ [10pt]
     \displaystyle\bigoplus_{k\geq 1}S(1+a_1+a_2+\cdots +a_{2k}) & \text{if $a_1= 0$.} 
\end{array}\right.
\end{align*}
Now, we choose $1\leq k_1<k_2<\cdots< k_t$ and $1\leq \ell_1<\ell_2< \cdots<  \ell_{t'}$ to be
\begin{align*}
& \mathsf{top}({\mathbb{I}(\mathbf{a})})=S(k_1)\oplus\cdots \oplus S(k_t), \\
& \mathsf{soc}({\mathbb{I}(\mathbf{a})})=S(\ell_1)\oplus\cdots \oplus S(\ell_{t'}).
\end{align*}
{Here, we put 
\[ \mathcal{T}_{\mathbf{a}}:=\{(k_{i_1},\ldots, k_{i_s})\in \mathbb{Z}^{s}\mid 1\leq i_1<\cdots <i_s\leq t,~s\in\mathbb{N}\}.
\]
}

A subquiver of $Q(\mathbf{a})$ of the form 
\[ \underbrace{\circ \too \circ \cdots \circ \too \circ}_{p_1\, \text{arrows}}\underbrace{ \oot \circ \cdots \circ \oot \circ}_{p_2  \, \text{arrows}} \]
is called a $(p_1,p_2)$-\textit{valley}.
For $(p_1,p_2)$-valley, we define a polynomial $\mathsf{val}_q(p_1,p_2)$ by
\[ \val_q(p_1,p_2):=\cl(Q(0,p_1,p_2)).\]
{Observe that the equation $\val_q(p_1,p_2)=\val_q(p_2,p_1)$ holds by \eqref{rotation closure poly.} and this can be calculated through the following.}

\begin{lem}
   For $(p_1,p_2)$-valley with $p_1\geq p_2$, we have
\[
   \val_q(p_1,p_2) = 1+\sum_{k=1}^{p_2+1}kq^k+(p_2+1)\sum _{k=p_2+2}^{p_1+1}q^{k}+\sum_{k=p_1+2}^{p_1+p_2+1}(p_1+p_2+2-k)q^k\] 
\end{lem}
\begin{proof}
This lemma follows from direct computation.
\end{proof}

Now, we define a sequence of pairs of integers as follows: 
\begin{enumerate}
\item[(i)] Compute {$\mathbf{a}-\mathbf{1}:=
\left\{\begin{array}{ll}
(a_1-1, a_2-1,\ldots, a_s-1) & \text{if $a_1\neq 0$,} \\
( a_2-1, a_3-1\ldots, a_s-1) & \text{if $a_1=0$.}
\end{array}\right.
$}
\item[(ii)] We put 
\[
(b_1,b_2,\ldots, b_{2m}):=\left\{\begin{array}{ll}
     (0,\mathbf{a}-\mathbf{1},0) & \text{if $a_1\neq 0$ and $s$ is even,} \\ [8pt]
     (0,\mathbf{a}-\mathbf{1}) & \text{if $a_1\neq 0$ and $s$ is odd,} \\ [8pt]
     ({\mathbf{a}-\mathbf{1}}, 0)& \text{ if $a_1=0$ and $s$ is even,} \\[8pt]
     ({\mathbf{a}-\mathbf{1}})& \text{ if $a_1=0$ and $s$ is odd.} \\[8pt]
\end{array}\right.
\]
\item[(iii)] 
{We set $\mathcal{J}_{\mathbf{a}}:=\{(b_1,b_2), (b_3,b_4),\ldots, (b_{2m-1},b_{2m})\}$.
}
\end{enumerate}

\begin{prop}\label{subrep-rad}
The number of $\ell$-dimensional subrepresentation of $\mathsf{rad}({\mathbb{I}(\mathbf{a})})$ coincides with the coefficient of $q^\ell$ of 
    \[
    \val_q(b_1,b_2)\cdot \val_q(b_3,b_4)\cdots\val_q(b_{2m-1},b_{2m}).
    \]
\end{prop}
\begin{proof}
{We show the case that $a_1>0$ and $s$ is even: the proof of other cases are similar.} 
In this case, $m=t+1$ and
the quiver $Q(\mathbf{a})$ is of the form:
\begin{align*} &\underbrace{1 \leftarrow \circ \cdots \circ\leftarrow \circ}_{b_2 \, \text{arrows}} \leftarrow  k_1\to \underbrace{\circ\to\circ \cdots\circ \to\circ}_{b_3 \, \text{arrows}}\underbrace{\leftarrow\circ \cdots \circ \leftarrow \circ}_{b_4  \, \text{arrows}} \leftarrow k_2\\
& \to\underbrace{\circ \to\circ \cdots\circ \to\circ}_{b_5 \, \text{arrows}}\underbrace{\leftarrow\circ \cdots \circ \leftarrow \circ}_{b_6  \, \text{arrows}} \leftarrow k_3\rightarrow\cdots
\leftarrow k_t\to\underbrace{\circ \to\circ \cdots\circ \to n}_{b_{2m-1} \, \text{arrows}}.
\end{align*}
This yields that $\mathsf{rad}({\mathbb{I}(\mathbf{a})})$ is decomposed as 
\[ 
\mathsf{rad}({\mathbb{I}(\mathbf{a})}))=\mathbb{I}[1,k_1-1]\oplus \mathbb{I}[k_1+1,k_2-1]\oplus \cdots\oplus \mathbb{I}[k_{t}+1,k_{t}-1]\oplus \mathbb{I}[k_{t}+1,n]. 
\]
{Thus, each subrepresentation $N\subset \mathsf{rad}(\mathbb{I}(\mathbf{a}))$ is the direct sum of subrepresentations $N_1\subset \mathbb{I}[1,k_1-1]$, $N_i\subset \mathbb{I}[k_i+1,k_{i+1}-1]$ $(i=1,\ldots, t-1)$, and $N_t\subset \mathbb{I}[k_t+1,n]$.}
{Since} the numbers of subrepresentations of $\mathbb{I}[1,k_1-1]$, $\mathbb{I}[k_i+1,k_{i+1}-1]$ $(i=1,\ldots, t-1)$, and $\mathbb{I}[k_t+1,n]$ are equal to 
{$\mathsf{val}_q(b_1,b_2)$, $\mathsf{val}_q(b_{2i+1},b_{2i+2})$ $(i=1,\ldots, t-1)$, and $\mathsf{val}_q(b_{2m-1},b_{2m})$, respectively, the assertion follows.}
\end{proof}

For each $k_i$ {$(i=1,\ldots, t)$}, a polynomial $\Delta_q(k_i)$ is defined as follows.
\begin{itemize}
\item[(i)] Suppose that $a_1>0$.  In this case, we define $\Delta_q(k_i)$ by
\[
\Delta_q(k_i):=\left\{\begin{array}{ll}
     q^{\ell_{i+1}-\ell_{i}+1}[\ell_{i}-k_{i-1}]_q[k_{i+1}-\ell_{i+1}]_q & \text{if $i\neq 1,t$,} \\ [8pt]
     q^{\ell_{2}}[k_{2}-\ell_{2}]_q& \text{ if $i=1$,} \\[8pt]
     q^{n-\ell_{t}+1}[\ell_t-k_{t-1}]_q& \text{ if $i=t$.}
\end{array}\right.
\]
\item[(ii)] Suppose that $a_1=0$. In this case, we define $\Delta_q(k_i)$ by
\[
\Delta_q(k_i):=\left\{\begin{array}{ll}
     q^{\ell_{i}-\ell_{i-1}+1}[\ell_{i-1}-k_{i-1}]_q[k_{i+1}-\ell_{i}]_q & \text{if $i\neq 1,t$,} \\ [8pt]
     q^{\ell_{1}}[k_{2}-\ell_{1}]_q& \text{ if $i=1$,} \\[8pt]
     q^{n-\ell_{t-1}+1}[\ell_{t-1}-k_{t-1}]_q& \text{ if $i=t$.} 
\end{array}\right.
\]
\end{itemize}
For each $k_i$, take a subset 
\[ \{(v_{2j-1}^{(k_i)},v_{2j}^{(k_i)})\mid j=1,2,\ldots, r_{k_i}\}\subset {\mathcal{J}_{\mathbf{a}}}
\]
such that any $(v_{2j-1}^{(k_i)},v_{2j}^{(k_i)})$-valley is not adjacent to vertex $k_i$.
Then, we set 
\[
\widetilde{\Delta}_q(k_i):=\Delta_q(k_i)\prod_{j=1}^{r_{k_i}}\val(v_{2j-1}^{(k_i)},v_{2j}^{(k_i)}).
\]

\begin{lem}\label{lem_delta}
The coefficient of $q^\ell$ of $\widetilde{\Delta}_q(k_i)$ coincides with 
the number of  $\ell$-dimensional subrepresentations $N$ of ${\mathbb{I}(\mathbf{a})}$ such that
$S(k_i)\in\mathsf{supp}(N)$, but $S(k_j)\notin\mathsf{supp}(N)$ for $i\neq j$.
\end{lem}
\begin{proof}
We only show the statement when $a_1>0$ and $s$ is even; the proofs in other cases are similar. 

{Let $N_{(k_i)}$ be the largest dimensional subrepresentation of $\mathbb{I}(\mathbf{a})$ such that $S(k_i)\in\mathsf{supp}(N)$, but $S(k_j)\notin\mathsf{supp}(N)$ for $i\neq j$.
It is sufficient to show that the coefficient of $q^\ell$ of $\widetilde{\Delta}_q(k_i)$ coincides with the number of $\ell$-dimensional subrepresentations of $N_{(k_i)}$.
Observe that $\mathbb{I}[k_{i-1}+1,k_{i+1}-1]\subset N_{(k_i)}$ and every subrepresentation of $N_{(k_i)}$ must have $\mathbb{I}[\ell_{i},\ell_{i+1}]$ as a subrepresentation whose dimension is $\ell_{i+1}-\ell_{i}+1$.}
Here, if $k_{i-1}$ (resp. $k_{i+1}$) is not in $Q(\mathbf{a})_0$, then we replace $k_{i-1}+1$ by $\ell_{1}$ (resp. $k_{i+1}-1$ by $\ell_{i+1}$).
Now, we consider an isomorphism
\[ \mathbb{I}[k_{i-1}+1,k_{i+1}-1]/\mathbb{I}[\ell_{i},\ell_{i+1}]\simeq \mathbb{I}[k_{i-1}+1, \ell_i-1]\oplus\mathbb{I}[\ell_{i+1}+1,k_{i+1}-1].
\]
{Since the number of $\ell$-dimensional subrepresentations of $\mathbb{I}[k_{i-1}+1, \ell_i-1]$ (resp. $\mathbb{I}[\ell_{i+1}+1,k_{i+1}-1]$) corresponds to the coefficient of $q^\ell$ of $[\ell_{i}-k_{i-1}]_q$ (resp. $[k_{i+1}-\ell_{i+1}]_q$),}
the number of $\ell$-dimensional subrepresentations $N'$ of $\mathbb{I}[k_{i-1}+1,k_{i+1}-1]$ such that $S(k_i)\in\mathsf{supp}(N')$ is the coefficient of $q^{\ell}$ of $\Delta_q(k_i)$.
Remaining subrepresentations that must be counted come from subrepresentations of $\mathsf{rad}({\mathbb{I}(\mathbf{a})})/(\mathbb{I}[k_{i-1}+1,k_{i+1}-1]/S(k_i))$.
Therefore, the assertion follows {from Proposition \ref{subrep-rad}}.
\end{proof}

Next, for two $k_{i_1}<k_{i_2}$, we define 
\[ 
\Delta_q(k_{i_1},k_{i_2}):=\left\{\begin{array}{ll}
\dfrac{\Delta_q(k_{i_1})\Delta_q(k_{i_2})}{q[\ell_{i_2}-k_{i_2-1}]_q[k_{i_1+1}-\ell_{i_1+1}]_q} & \text{if $i_2=i_1+1$,}\\[15pt]
\Delta_q(k_{i_1})\Delta_q(k_{i_2}) & \text{otherwise.}
\end{array}\right.
\]
Inductively, for $k_{i_1}<k_{i_2}<\cdots <k_{i_r}$, we define
\[ \Delta_q(k_{i_1},k_{i_2},\ldots, k_{i_r}):=
\left\{\begin{array}{ll}
    \dfrac{\Delta_q(k_{i_1}, k_{i_2},\ldots, k_{i_r-1})\Delta_q(k_{i_r})}{q[\ell_{i_r}-k_{i_r-1}]_q} & \text{if $i_r=i_{r-1}+1$,}\\[15pt]
\Delta_q(k_{i_1}, k_{i_2},\ldots, k_{i_r-1})\Delta_q(k_{i_r}) & \text{otherwise.}
\end{array}\right.
\]
Take a subset
\[ \{(v_{2j-1}^{(k_{i_1},\ldots, k_{i_r})},v_{2j}^{(k_{i_1},\ldots, k_{i_r})})\mid j=1,2,\ldots, r_{(k_{i_1},\ldots, k_{i_r})}\}\subset {\mathcal{J}_{\mathbf{a}}}
\]
such that any $(v_{2j-1}^{(k_{i_1},\ldots, k_{i_r})},v_{2j}^{(k_{i_1},\ldots, k_{i_r})})$-valley is not adjacent to one of vertices $k_{i_1},\ldots, k_{i_r}$.
Then, {for $r=1,\ldots, t$,} we set 
\[
\widetilde{\Delta}_q(k_1,\ldots, k_r):=\Delta_q(k_1,\ldots,k_r)\prod_{j=1}^{r_{(k_1,\ldots,k_r)}}\val(v_{2j-1}^{(k_1,\ldots,k_r)},v_{2j}^{(k_1,\ldots,k_r)}).
\]

\begin{thm}
    The equation
    \begin{equation}\label{subrep Q(a)}
\cl({Q(\mathbf{a})})=\prod_{{(b_i,b_{i+1})\in\mathcal{J}_{\mathbf{a}}}}\val_q(b_i,b_{i+1})+\sum_{{(k_{i_1},\ldots, k_{i_s})\in \mathcal{T}_{\mathbf{a}}}}\widetilde{\Delta_q}(k_{i_1},\ldots, k_{i_s})
    \end{equation}
    holds.
\end{thm}
\begin{proof}
{By Proposition \ref{subrep-rad}, the first term of the right-hand side of \eqref{subrep Q(a)} counts  $\ell$-dimensional subrepresentations of $\mathsf{rad}(\mathbb{I}(\mathbf{a}))$.
Therefore, counting the cases where each $S(k_i)$ $(i=1,\ldots, t)$ belongs to the support is sufficient.}
By the proof of Lemma \ref{lem_delta}, the number of $\ell$-dimensional subrepresentations $N$ of ${\mathbb{I}(\mathbf{a})}$ such that $S(k_{i_1}), \ldots, S(k_{i_r})\in\mathsf{supp}(N)$ but $S_j\not\in \mathsf{supp}(N)$ for $j\neq k_{i_1},\ldots,k_{i_r}$ is the coefficient of $q^\ell$ of  $\widetilde{\Delta}_q(k_{i_1},\ldots, k_{i_r})$.
Thus, the assertion follows.
\end{proof}

\begin{exmp}
(1) Let $\mathbf{a}=(1,3,1,1)$. Then, the quiver $Q(\mathbf{a})$ is of the form
\[ 
Q(\mathbf{a})=1\oot 2\too 3\too 4\too 5\oot 6\too 7,
\]
and $((b_1,b_2),(b_3,b_4),(b_5,b_6))=((0,0), (2,0), (0,0))$. So, we compute
\begin{align*}
    \val_q(0,0)\val_q(2,0)\val_q(0,0)& =q^5+3q^4+4q^3+4q^2+3q+1,\\
   \widetilde{\Delta}_q(2) = q^5[1]_q[1]_q\val_q(0,0) & =q^6+q^5,  \\
   \widetilde{\Delta}_q(6) = q^3[3]_q\val_q(0,0) & =q^6+2q^5+2q^4+q^3,  \\
   \widetilde{\Delta}_q(2,6) =\dfrac{\Delta_q(2)\Delta_q(6)}{q[3]_q[1]_q} &=q^7.  
\end{align*}
Thus, we have
\[
\cl(Q(\mathbf{a}))=q^7+2q^6+4q^5+5q^4+5q^3+4q^2+3q+1.
\]

(2) Let $\mathbf{a}=(0,3,1,5,1)$. Then, the quiver $Q(\mathbf{a})$ is of the form
\[ 
Q(\mathbf{a})=1\too 2\too 3\too 4\oot 5\too 6\too 7\too8\too 9\too 10\oot 11 ,
\]
and $((b_1,b_2),(b_3,b_4))=((2,0), (4,0))$. So, we compute
\begin{align*}
    \val_q(2,0)\val_q(4,0) & =q^8+2q^7+3q^6+4q^5+4q^4+4q^3+3q^2+2q+1,\\
   \widetilde{\Delta}_q(1) = q^4\val_q(4,0) & =q^9+q^8+q^7+q^6+q^5+q^4, \\
   \widetilde{\Delta}_q(5) = q^7[3]_q& =q^9+q^8+q^7,  \\
   \widetilde{\Delta}_q(11) = q^2[5]_q\val_q(2,0)& =q^9+2q^8+3q^7+4q^6+4q^5+3q^4+2q^3+q^2, \\
   \widetilde{\Delta}_q(1,5) =\dfrac{\Delta_q(1)\Delta_q(5)}{q[3]_q} &=q^{10},  \\
   \widetilde{\Delta}_q(1,11) =\Delta_q(1)\Delta_q(11) &=q^{10}+q^9+q^8+q^7+q^6,  \\
   \widetilde{\Delta}_q(5,11) =\dfrac{\Delta_q(5)\Delta_q(11)}{q[5]_q} &=q^{10}+q^9+q^8,  \\
   \widetilde{\Delta}_q(1,5,11) =\dfrac{\Delta_q(1)\Delta_q(5)\Delta_q(11)}{q^2[3]_q[5]_q} &=q^{11}.  
\end{align*}
Thus, we have
\[
\cl(Q(\mathbf{a}))=q^{11}+3q^{10}+5q^9+7q^8+8q^7+9q^6+9q^5+8q^4+6q^3+4q^2+2q+1.
\]
\end{exmp}

\section{Special values of the \texorpdfstring{$q$}{q}-deformed rational numbers} \label{sec-7}
In \cite[Proposition~1.8]{MO}, it is shown that both $\cS_\alpha(-1)$ and $\cR_\alpha(-1)$ belong to  $\{0, \pm 1\}$. From this, we see that for an irreducible fraction $\frac{r}s$, $s$ is even if and only if $\cS_{\frac{r}s}(q)$ is divisible by $[2]_q=1+q$.  
In this section, we extend this observation. Set $$\omega :=\dfrac{-1+\sqrt{-3}}2.$$ 

\begin{thm}\label{mod 3}
For a rational number  $\alpha$,  we have $\cR_\alpha(\omega), \cS_\alpha(\omega) \in \{0, \pm 1, \pm \omega, \pm \omega^2 \}$.  
\end{thm}

\begin{proof} 
First, we assume that $\alpha >1$, and write $\alpha=[[c_1, \ldots, c_l]]$. By Proposition~\ref{MO20 prop. 4.3 and 4.9}, we have  
$$
\begin{pmatrix}\cR_\alpha(\omega)\\ \cS_\alpha(\omega) \end{pmatrix}=
\left(  M_q^-(c_1) M_q^-(c_2) \cdots M_q^-(c_l)\right)|_{q=\omega} \begin{pmatrix} 1\\0 \end{pmatrix}.
$$ 
It is easy to check that 
$M_q^-(c)|_{q=\omega}$ for a positive integer $c$ is one of the following forms:
\[ M_q^-(c)|_{q=\omega}=\left\{
\begin{array}{ll}
X:=\begin{pmatrix}0&-\omega^2\\1&0 \end{pmatrix}& \text{if $c \equiv 0 \pmod{3}$,} \\[15pt]
Y:= \begin{pmatrix}1&-1\\1&0 \end{pmatrix} & \text{if $c \equiv 1 \pmod{3}$,} \\[15pt]
Z:=\begin{pmatrix}-\omega^2&-\omega\\1&0 \end{pmatrix} & \text{if $c \equiv 2 \pmod{3}$.}
\end{array}
\right. \]

Let $G$ be the subgroup of $\mathsf{GL}(2,\CC)$ generated by $X,Y$ and $Z$. A direct computation shows that the $X^{12}=Y^6=Z^3=E_2$. 
Set 
\begin{equation}\label{set A}
A:=\left\{ \zeta \! \begin{pmatrix}1\\0\end{pmatrix}, 
\zeta \!  \begin{pmatrix}1\\-\omega\end{pmatrix},  
\zeta \! \begin{pmatrix}1\\1\end{pmatrix}, 
\zeta \! \begin{pmatrix}0\\1 \end{pmatrix}   \, \middle | \, \zeta = \pm 1, \pm \omega, \pm \omega^2
\right\} \subset \CC^2.
\end{equation}
Then, easy calculation shows that $A$ is closed under the natural action of $G$. 
Thus, for any $W \in G$, all entries of $W$, especially $\cR_\alpha(\omega)$ and  $\cS_\alpha(\omega)$ for $\alpha >1$, belong to the set $\{0, \pm 1, \pm \omega, \pm \omega^2 \}$.

Let us consider the case $\alpha \le 1$. By \eqref{right-sift}, we have 
\begin{equation}\label{alpha => alpha+1}
\begin{pmatrix}\cR_\alpha(\omega)\\ \cS_\alpha(\omega)\end{pmatrix} = \begin{pmatrix}\omega^2&-\omega^2\\0&1 \end{pmatrix} \begin{pmatrix}
\cR_{\alpha+1}(\omega)\\ \cS_{\alpha+1}(\omega)\end{pmatrix}. 
 \end{equation}
It is easy to check that the set $A$ is closed under the multiplication of 
$\begin{pmatrix}\omega^2&-\omega^2\\0&1 \end{pmatrix}$, so we can show that $\begin{pmatrix}
\cR_{\alpha+1}(\omega)\\ \cS_{\alpha+1}(\omega)\end{pmatrix} \in A$ implies $\begin{pmatrix}
\cR_{\alpha}(\omega)\\ \cS_{\alpha}(\omega)\end{pmatrix}\in A$.
Since $\alpha+n >1$ for sufficiently large $n$, the desired assertion follows from repeated use of the above implication. 
\end{proof}

Since the leading coefficient of $[n]_q=1+q+\cdots +q^{n-1}$ is 1, when we divide $f(q) \in \ZZ[q]$ by $[n]_q$, the quotient and the remainder belong to $\ZZ[q]$.  
It is clear that if $\cS_{\frac{r}s}(q)$ can be divided by $[3]_q=1+q+q^2$, then  $s=\cS_{\frac{r}s}(1)$ is a multiple of 3. The following states that the converse is also true.   

\begin{cor}\label{3ZZ}
The following assertions hold.
\begin{enumerate}
\item[(1)] If  $s=\cS_{\frac{r}s}(1)$ is a multiple of $3$, $\cS_{\frac{r}s}(q)$ can be divided by  $[3]_q$. 
Moreover, for an irreducible fraction $\frac{r}s$, we have 
$$
\cS_{\frac{r}s}(\omega)=\begin{cases}
0 & \text{if $s \equiv 0 \pmod{3}$,} \\
1, \omega, \omega^2 & \text{if $s \equiv 1 \pmod{3}$,} \\
-1, -\omega, -\omega^2  & \text{if $s \equiv 2 \pmod{3}$.} 
\end{cases}
$$
\item[(2)] Similarly, we have
$$
\cR_{\frac{r}s}(\omega)=\begin{cases}
0 & \text{if $r \equiv 0 \pmod{3}$,} \\
1, \omega, \omega^2 & \text{if $r \equiv 1 \pmod{3}$,} \\
-1, -\omega, -\omega^2  & \text{if $r\equiv 2 \pmod{3}$.} 
\end{cases}
$$
\end{enumerate}
\end{cor}

\begin{proof}
(1) Note that $\omega^2=-(w+1)$. 
By Theorem~\ref{mod 3}, the remainder of the polynomial $\cS_{\frac{r}s}(q)$ divided by $[3]_q$ is $aq +b$ for $a, b \in \{0, \pm 1\}$ with $(a,b) \ne (1,-1), (-1,1)$. Since $s=\cS_{\frac{r}s}(1) \equiv a+b \pmod{3}$, the assertion easily follows. 

(2) While $\cR_{\frac{r}s}(q) \in \ZZ[q,q^{-1}]$ has terms of negative degree for $\frac{r}s <0$,  we have $f(q):=q^{3n}\cR_{\frac{r}s}(q) \in \ZZ[q]$ for $n \gg 0$. Since $f(1)=\cR_{\frac{r}s}(1)=r$ and $f(\omega) =\cR_{\frac{r}s}(\omega)$, we can use the argument of the proof of (1).  
\end{proof}

\begin{exmp}
Even if we fix $s$, $\cS_{\frac{r}s}(\omega)$ depends on $r$. For example, we have $\cS_{\frac{12}{11}}(\omega)=-\omega^2$, $\cS_{\frac{13}{11}}(\omega)=-\omega$,   $\cS_{\frac{14}{11}}(\omega)=-\omega^2$, $\cS_{\frac{15}{11}}(\omega)=-1$, and so on.   
\end{exmp}

\begin{cor}\label{R(w)=S(w)}
For an irreducible fraction $\frac{r}s$, 
$s \equiv r \pmod 3$ if and only if $\cR_{\frac{r}s}(\omega)=\cS_{\frac{r}s}(\omega)$. 
\end{cor}

\begin{proof}
By \eqref{right-sift}, we have
$
\cR_{\frac{r}s-1}(q)=q^{-1}(\cR_{\frac{r}s}(q) -\cS_{\frac{r}s}(q)). 
$
By Corollary~\ref{3ZZ}, we have 
$$
\cR_{\frac{r}s}(\omega)=\cS_{\frac{r}s}(\omega) \quad\Longleftrightarrow\quad \cR_{\frac{r-s}{s}}(\omega)=0 \quad\Longleftrightarrow\quad \text{$r-s$ is a multiple of 3.} $$
So we are done.  
\end{proof}

In the rest of this section, $i$ means $\sqrt{-1}$. 

\begin{prop}\label{S(i)}
We have $\cR_\alpha(i), \cS_\alpha(i) \in \{0, \pm 1, \pm i , \pm (1 + i), \pm(1- i) \}$, and the remainder of $\cS_\alpha(q)$ divided by  $q^2+1$ is $aq+b$ for $a,b \in \{0, \pm 1 \}$. 
\end{prop}

\begin{proof}
It suffices to show the first assertion.
The proof is similar to that of Theorem~\ref{mod 3}. First, assume that $\alpha > 1$.  
It is easy to check that $M_q^-(c)|_{q=i}$ is of the form 
\[ M^-_q(c)|_{q=i}=\left\{
\begin{array}{ll}
X_0:=\begin{pmatrix}0&i\\1&0 \end{pmatrix}& \text{if $c \equiv 0 \pmod{4}$,} \\[15pt]
X_1:= \begin{pmatrix}1&-1\\1&0 \end{pmatrix} & \text{if $c \equiv 1 \pmod{4}$,} \\[15pt]
X_2:= \begin{pmatrix}1+i&-i\\1&0 \end{pmatrix} & \text{if $c \equiv 2 \pmod{4}$,} \\[15pt]
X_3:=\begin{pmatrix}i&1\\1&0 \end{pmatrix} & \text{if $c \equiv 3 \pmod{4}$.}
\end{array}
\right. \]

 A direct computation shows that $X_0^8=X_1^6=X_2^4=X_3^{12}=E_2$. 
Let $G^{\prime}$ be the subgroup of $\mathsf{GL}(2,\CC)$ generated by $X_0,X_1,X_2$ and $X_3$.
The set 
$$B:=\left\{ 
\zeta \!  \begin{pmatrix}1\\0\end{pmatrix},
\zeta \!  \begin{pmatrix}0\\1\end{pmatrix}, 
\zeta \!  \begin{pmatrix}1\\1\end{pmatrix},
\zeta \!  \begin{pmatrix}i\\1\end{pmatrix}, 
\zeta \!  \begin{pmatrix}1+i\\1\end{pmatrix}, 
\zeta \!  \begin{pmatrix}1\\1-i\end{pmatrix}
\, \middle | \, \zeta = \pm 1, \pm i
\right\}.$$
is closed under the natural action of $G^{\prime}$. 
Hence all entries of any element  in $G'$ belong to $\{0, \pm 1, \pm i, \pm (1 + i), \pm(1- i) \}$.  
Since $\cR_\alpha(i)$ and $\cS_\alpha(i)$ are  entries of a suitable element of $G'$, we are done. 

For the case $\alpha \le 1$, we can use the same argument as the last part of the proof of Theorem~\ref{mod 3}. 
\end{proof}

\begin{thm}\label{mod 4}
For an irreducible fraction $\frac{r}s$, the following are equivalent. 
\begin{itemize}
\item[(1)] $s$ is a multiple of $4$, 
\item[(2)] $\cS_{\frac{r}s}(q)$ is divisible by $[4]_q=q^3+q^2+q+1$, 
\item[(3)] $\cS_{\frac{r}s}(q)$ is divisible by $q^2+1$. 
\end{itemize}
\end{thm}

\begin{proof}
(1) $\Rightarrow$ (2) : 
Let $g(q) \in \ZZ[q]$ be the remainder of $\cS_{\frac{r}s}(q)$ divided by $1+q^2$, that is, 
$$\cS_{\frac{r}s}(q)=f(q) \cdot (1+q^2)+g(q) \qquad (f(q) \in \ZZ[q],\, \mathsf{deg}(g) \leq 1).$$
Since $[4]_q=(1+q)(1+q^2)$ and $\cS_{\frac{r}s}(-1)=0$ by \cite[Proposition~1.8]{MO}, it suffices to show that  $\cS_{\frac{r}s}(q)$ is divisible by $1+q^2$ (equivalently, $g(q) =0$). For the contradiction, assume that $g(q) \ne 0$. 
Proposition~\ref{S(i)} states that $g(q) = \pm 1, \pm q,
\pm(1+q),  \pm(1- q)$.  However, since $g(1)=s-2f(1)$ and $s$ is a multiple of 4, $g(1)$ is even, and hence $g(q)\ne \pm 1, \pm q$. Finally, we have $g(q) = \pm(1+q),  \pm(1- q)$.

In what follows, for $f(q) \in \ZZ[q]$, $(f(q))$ denotes the ideal of $\ZZ[q]$ generated by $f(q)$, and $\ZZ[q]/(f(q))$ denotes the quotient ring. For the canonical surjections  $\pi_1 : \ZZ[q]  \to \ZZ[q]/(1+q)$ and  $\pi_2 : \ZZ[q] \to \ZZ[q]/(1+q^2)$ 
(if there is no danger of confusion, we denote $\pi_i(f(q))$ by $\overline{f(q)}$), 
consider the ring homomorphism
$$
\phi: \mathbb{Z}[q] \ni f(q) \longmapsto (\pi_1(f(q)), \pi_2(f(q)))  \in \left(\mathbb{Z}[q]/\left(1+q\right)\right) \times \left(\mathbb{Z}[q]/\left(1+q^2\right)\right).
$$
Since $\ZZ[q]$ is a UFD, and $1+q$ and $1+q^2$ are coprime, we have $\mathsf{ker}(\phi) =([4]_q)$.  
In the present situation, we have  $$\phi(\cS_{\frac{r}s}(q))=( \overline{0}, \overline{g(q)}).$$ 
Recall that  $g(q) = \pm(1+q),  \pm(1- q)$, but we have 
$$\phi(\pm (1+q))=( \overline{0}, \pm \overline{(1+q)}) \quad \text{or} \quad \phi(\pm (q+q^2))=( \overline{0}, \mp\overline{(1-q)}).$$
Hence, we have either 
$$\pm (1+q) -\cS_{\frac{r}s}(q) \in ([4]_q) \quad \text{or} \quad \pm (q+q^2) -\cS_{\frac{r}s}(q) \in ([4]_q).$$
In both cases, $\pm 2 -\cS_{\frac{r}s}(1) \in 4\ZZ$, and it means that $\cS_{\frac{r}s}(1) \equiv 2 \pmod{4}$. 
It contradicts the assumption that $\cS_{\frac{r}s}(1) \in 4\ZZ$. 

(2) $\Rightarrow$ (3) : Obvious. 

(3) $\Rightarrow$ (2) :  If $\cS_{\frac{r}s}(q)$ is divisible by $1+q^2$, then there is some $f(q) \in \ZZ[q]$ such that $\cS_{\frac{r}s}(q)=(1+q^2)f(q)$. It follows that $s = \cS_{\frac{r}s}(1)=2f(1)$ is even, and hence $\cS_{\frac{r}s}(q)$ is also divisible by $1+q$. Since $[4]_q=(1+q)(1+q^2)$, the assertion follows. 
\end{proof}

The next result can be proved by an argument similar to the corresponding results for $q=\omega$. 

\begin{cor}\label{mod 4-i}
The following assertions hold.
\begin{enumerate}
\item[(1)] We have
$$\cS_{\frac{r}s}(i) =\begin{cases}
0 &\text{if $s \equiv 0 \pmod{4}$},\\
\pm(1+i), \pm(1-i) &  \text{if $s \equiv 2 \pmod{4}$},\\
\pm 1, \pm i &  \text{if $s \equiv 1 \pmod{2}$},
\end{cases}$$
and 
$$\cR_{\frac{r}s}(i) =\begin{cases}
0 & \text{if $r \equiv 0 \pmod{4}$},\\
\pm(1+i), \pm(1-i) &  \text{if $r \equiv 2 \pmod{4}$},\\
\pm 1, \pm i &  \text{if $r \equiv 1 \pmod{2}$.}
\end{cases}$$
\item[(2)] For an irreducible fraction $\frac{r}{s}$, we have  $s\equiv r\pmod4$ if and only if $\cR_{\frac{r}s}(i)=\cS_{\frac{r}s}(i)$.  
\end{enumerate}
\end{cor}

\begin{exmp}
It is clear that the analog of Corollaries~\ref{3ZZ} and \ref{mod 4-i} does not hold for primitive $n$-th roots of unity with $n \ge 5$. In fact, since $\cS_{\frac{7}5}(q)=q^3+2q^2+q+1$, we have $\cS_{\frac{7}5}(\zeta)\ne 0$, where $\zeta$ is a primitive 5th root of unity (i.e., a root of $q^4+q^3+q^2+q+1$). Moreover, using a computer system, we see that $\cS_{\frac{37}{35}}(q)$ is irreducible over $\QQ$, while 35 is a composite number.
\end{exmp}

\begin{conj}\label{irreducibility}
If $p$ is a prime integer, then $\cS_{\frac{a}{p}}(q)$ is irreducible over $\QQ$.  
\end{conj}

Using the computer program Maple, we checked the conjecture for prime numbers up to 
739. The following is another piece of evidence.

\begin{thm}
Let $p$ be a prime integer. 
If  $\cS_{\frac{a}p}(q)$ is reducible in $\QQ[q]$ (i.e., Conjecture~\ref{irreducibility} does not hold), all of its factors have degree at least $7$.  
\end{thm}

\begin{proof}
Consider the factorization  
$$\cS_{\frac{a}p}(q)=\prod_{j=1}^k f_j(q)$$ in the polynomial ring $\QQ[q]$.
It is a classical result that we can take $f_j(q)$  from $\ZZ[q]$ for all $j$. 
Assume that $k \ge 2$. Since $f_j(1) \in \ZZ$ for all $j$ and $p=\cS_{\frac{a}p}(1) =\prod_{j=1}^k f_j(1)$ is a prime number, we may assume that $f_1(1)=p$ and $f_j(1)=1$ for all $j \geq 2$. 

Since both the leading coefficient and constant term of $\cS_{\frac{a}p}(q)$ are 1, those of $f_j(q)$ are $\pm 1$.  
Since all coefficients of $\cS_{\frac{a}p}(q)$ are positive, if $q=\alpha$ is  a {\it real} root of the equation $\cS_{\frac{a}p}(q)=0$ then $\alpha < 0$. Clearly, the same is true for each $f_j(q)$, so both the leading coefficient and constant term of $f_j(q)$ are 1 (note that $f_j(1) >0$ now).

If $p=2,3$, the assertion is clear. So we may assume that $p \ge 5$. Since $p=\cS_{\frac{a}p}(1)$ is odd, $\cS_{\frac{a}p}(-1)=\prod_{j=1}^k f_j(-1)=\pm 1$. Since $f_j(-1) \in \ZZ$ for all $j$, we have $f_j(-1) =\pm 1$, and hence the remainder of $f_j(q)$  divided by $q+1$ is $\pm 1$. 
Similarly,  we have $\cS_{\frac{a}p}(i)=\prod_{j=1}^k f_j(i)=\pm 1, \pm i$ by Corollary~\ref{mod 4-i}. 
Since $f_j(i) \in \ZZ[i]$ for all $j$, we have $f_j(i) =\pm 1, \pm i$, and the remainder of $f_j(q)$  divided by $q^2+1$ is $\pm 1, \pm q$. Since $p=\cS_{\frac{a}p}(1)$ is not a multiple of 3, $\cS_{\frac{a}p}(\omega)=\prod_{j=1}^k f_j(\omega)=\pm 1, \pm \omega, \pm \omega^2$. 
Since $f_j(\omega) \in \ZZ[\omega]$ for all $j$, we have $f_j(\omega) =\pm 1, \pm \omega, \pm \omega^2$ by Corollary~\ref{3ZZ}, and the remainder of $f_j(q)$  divided by $q^2+q+1$ is $\pm 1, \pm q, \pm (1+q)$.

Set $g(q)=q(q+1)(q^2+1)(q^2+q+1)$, and consider the natural ring homomorphism 
$$\Psi: \ZZ{[q]}/(g(q)) \too \ZZ[q]/(q) \times  \ZZ[q]/(q+1) \times  \ZZ[q]/(q^2+1) \times  \ZZ[q]/(q^2+q+1).$$
Since $\ZZ[q]$ is a UFD, $\Psi$ is injective. 
Let us find polynomials in $\ZZ[q]$ whose images under $\Psi$ are characteristic.

For $t(q) := (q+1)(q^2+1)(q^2+q+1)$, we have $\Psi(t(q))=(\one, \zero, \zero, \zero)$ and $t(1)=12$. For
$$
\begin{array}{ll}
u_1(q):=q(q^2+q+1), & u_2(q):=q^2(q^2+q+1), \\ 
u_3(q):=q^3(q^2+q+1), & u_4(q):=q(q^2+q+1)^2,
\end{array}$$
we have 
$$
\begin{array}{ll}
\Psi(u_1(q))=(\zero, -\one, -\one, \zero), & \Psi(u_2(q))=(\zero, \one, -\qb, \zero),\\ [5pt]
\Psi(u_3(q))=(\zero, -\one, \one, \zero),  &  \Psi(u_4(q))=(\zero, -\one, -\qb, \zero),
\end{array}$$ and $u_k(1)=3$ for $k= 1,2,3$, $u_4(1)=9$. 
For 
$$v_1(q):=q(q+1)(q^2+1),  \quad v_2(q):=q(q+1)^2(q^2+1),  \quad v_3(q):=q^2(q+1)(q^2+1),$$
we have $$\Psi(v_1(q))=(\zero, \zero, \zero, \qb), \quad \Psi(v_2(q))=(\zero, \zero, \zero, -\one), \quad \Psi(v_3(q))=(\zero, \zero, \zero, -\one-\qb)$$  and  $v_1(1)=v_3(1)=4$, $v_2(1)=8$.

The possible values of $\Psi(f(q))$ have been determined above, and the leading coefficient of $f_j(q)$ is 1. Hence, if $\mathsf{deg}~f_j(q) \le 6$, we have 
$$f_j(q)= c_1 g(q)+ t(q) + c_2 u_k(q)+ c_3 v_l(q)$$ 
for some $c_1 =0,1$, $c_2, c_3 =\pm 1$, $k=1,\ldots, 4$ and $l=1,2,3$.
If $j \ge 2$, $f_j(q)$ must satisfy the following conditions:  
\begin{itemize}
\item $f_j(q) \ne 1$ and $f_j(1)=1$. 
\item The leading coefficient is 1.
\end{itemize}
However, easy calculation shows that no choice of $c_1, \ldots, c_3, j,k$ satisfies these conditions. 
Finally, we consider $f_1(q)$. We have
$$p =f_1(1) \le g(1)+t(1) + u_k(1)+ v_l(1)\le 12+12+9+8=41.$$
However, Conjecture~\ref{irreducibility} has been checked in this range by using Maple. 
\end{proof}

\section{Application to Jones polynomials of rational knots} \label{sec-8}
Using the results in the previous section, we study the special values of the Jones polynomial $V_\alpha(t)$ and the normalized one $J_\alpha(q)$ of a rational link $L(\alpha)$.  

For a general link $L$, it is a classical fact 
that $$V_L(1)=(-2)^{c(L)-1},$$
where $c(L)$ is the number of the components of $L$. 
On the other hand, for an irreducible fraction $\frac{r}s$, it is well-known that $c(L(\frac{r}s))=1,2$, and $c(L(\frac{r}s))=1$ if and only if $r$ is odd. 
Hence we have 
$$
V_{\frac{r}s}(1)=
\begin{cases}
-2 &\text{if $r$ is even},\\
1 & \text{if $r$ is odd}. 
\end{cases}
$$
We can explain this equation using $q$-deformed rationals. 

Recall the equation~\eqref{eq2-6}, which states that the normalized Jones polynomial $J_\alpha(q)$ of a rational link $L(\alpha)$ can be computed by the following formula:
\[
J_\alpha(q)=q \cdot \cR_\alpha(q)+(1-q) \cdot \cS_\alpha(q).
\]

By an argument similar to the previous section, we can show that 
$$\begin{pmatrix}\cR_{\frac{r}s}(-1)\\\cS_{\frac{r}s}(-1)\end{pmatrix}=\pm \begin{pmatrix}1\\0\end{pmatrix}, \pm  \begin{pmatrix} 0 \\1\end{pmatrix}, \pm \begin{pmatrix}1\\1\end{pmatrix}$$
(this is a refinement of \cite[Proposition~1.8]{MO}). Hence we have 
\begin{equation}\label{V(1)}
|V_{\frac{r}s}(1)|=|J_{\frac{r}s}(-1)|=
\begin{cases}
2 & \text{if $r$ is even},\\
1 & \text{if $r$ is odd}. 
\end{cases}
\end{equation}

Next, we will consider the special values of $J_\alpha(q)$ at $q=i, \omega, -\omega$. Many parts of the following results should be well-known, but we are interested in the relation to $q$-deformed rationals.

\begin{thm}\label{J(w)}
For an irreducible fraction $\frac{r}s >1$, we have 
$$
J_{\frac{r}{s}}(\omega) \in \{\pm 1, \pm \omega, \pm \omega^2 \},
$$
if $r$ is not a multiple of $3$, and 
$$
J_{\frac{r}{s}}(\omega) \in \{\pm (1-\omega), \pm \omega(1-\omega), \pm \omega^2(1-\omega) \},  
$$
if $r$ is a multiple of $3$. In particular, 
\begin{equation}\label{|J(w)|}
|V_{\frac{r}{s}}(-\omega) |=|J_{\frac{r}{s}}(\omega) |=\begin{cases}
\sqrt{3} & \text{if $r$ is a multiple of $3$,} \\
1 & \text{otherwise. }
\end{cases}
\end{equation}
\end{thm}

\begin{proof}
The assertion easily follows from (the proof of) Theorem~\ref{mod 3}. By \eqref{eq2-6}, we have
$$J_{\frac{r}{s}}(\omega)=\begin{pmatrix}\omega &1 -\omega \end{pmatrix} \begin{pmatrix}\cR_{\frac{r}s}(\omega)\\\cS_{\frac{r}s}(\omega)\end{pmatrix}
\quad \text{and} \quad \begin{pmatrix}\cR_{\frac{r}s}(\omega)\\\cS_{\frac{r}s}(\omega)\end{pmatrix} \in A,
$$ 
where $A$ is the set given in \eqref{set A}. 
So the assertion follows.  
\end{proof}

\begin{rem}
For a general link $L$, Lickorish and Millett (\cite[Theorem~3]{LM}) showed that 
\begin{equation}\label{Lickorish-Millett}
V_L(-\omega)=\pm i^{c(L)-1}(\sqrt{3}i)^d,
\end{equation}
where $d = \mathsf{dim} H_1(\Sigma(L);\ZZ_3)$ with $\Sigma(L)$ the double cover of the $3$-sphere $\mathbb{S}^3$ branched over $L$.

By \eqref{def normalized Jones}, we have 
$$V_{\frac{r}s}(-\omega)= \pm (-\omega)^{h} J_{\frac{r}s}(\omega^{-1})$$
(note that $\omega^{-1}=\omega^2=\overline{\omega}$). 
Hence, comparing \eqref{Lickorish-Millett} with \eqref{|J(w)|}, we have 
$$
\mathsf{dim} H_1(\Sigma(L(r/s));\ZZ_3)=\begin{cases}
1 & \text{if $r$ is a multiple of 3,} \\
0 & \text{otherwise. }
\end{cases}
$$
\end{rem}

The next result can be proved similarly to Theorem~\ref{J(w)}, but we use Proposition~\ref{S(i)} this time.    

\begin{thm}\label{J(i)}
For an irreducible fraction $\frac{r}s >1$, we have 
$$
J_{\frac{r}{s}}(i) = \begin{cases} 
0 & \text{if $r \equiv 2 \pmod{4}$,} \\
\pm (1+i), \pm(1-i) & \text{if $r \equiv 0 \pmod{4}$,} \\
\pm 1, \pm i & \text{if $r \equiv 1,3 \pmod{4}$.}
\end{cases} 
$$
\end{thm}

\begin{rem}
For a general link $L$, Murakami \cite{M} (see also \cite[Theorem~1]{LM}) showed that 
$$
V_L(i) =\begin{cases}
(-\sqrt{2})^{c(L)-1}(-1)^{\Arf(L)} & \text{if $\Arf(L)$ exists,} \\
0 & \text{otherwise.}
\end{cases}
$$
Comparing this equation with Theorem~\ref{J(i)}, we see that $\Arf(L{(\frac{r}{s})})$ exists if and only if $r \not \equiv 2 \pmod{4}$. {We were unable to find this statement
in literature, but it must be possible to prove it directly. }
\end{rem}

For a general link $L$, it is known that $V_L(\omega)=(-1)^{c(L)-1}$. Hence, for a rational link $L(\frac{r}s)$, we have $V_{\frac{r}s}(\omega)=(-1)^{r-1}$ and hence 
\begin{equation}\label{J(-w)}
J_{\frac{r}s}(-\omega) \in \{\pm 1, \pm \omega, \pm \omega^2\}. 
\end{equation}
We can give a new interpretation to this equation using $q$-deformed rationals. 
Note that $ M_q^-(c)|_{q=-\omega}$ is of the form 
\[
\begin{array}{llll}
X_0:=\begin{pmatrix}0&\omega^2\\1&0 \end{pmatrix}  &  \text{if $c \equiv 0 \pmod{6}$,}  \\[15pt] 
X_1:=\begin{pmatrix}1&-1\\1&0 \end{pmatrix} & \text{if $c \equiv 1 \pmod{6}$,} \\[15pt] 
X_2:=\begin{pmatrix}1-\omega&\omega\\1&0 \end{pmatrix} &  \text{if $c \equiv 2 \pmod{6}$,} \\[15pt] 
X_3:=\begin{pmatrix}1-\omega+\omega^2&-\omega^2\\1&0 \end{pmatrix}  &  \text{if $c \equiv 3 \pmod{6}$,} \\[15pt] 
X_4:=\begin{pmatrix}-\omega+\omega^2&1 \\1&0 \end{pmatrix} & \text{if $c \equiv 4 \pmod{6}$,} \\[15pt] 
X_5:=\left( \begin{array}{cccc}\omega^2&-\omega \\1&0 \end{array}\right)  &   \text{if $c \equiv 5 \pmod{6}$.}
\end{array}
\]
By \eqref{eq2-6}, for $\alpha=[[c_1, \ldots, c_l]]$, we have 
$$\begin{pmatrix}
J_\alpha(-\omega)&  *
\end{pmatrix}=
\begin{pmatrix}
-\omega&  1+\omega
\end{pmatrix} \cdot \left(  M_q^-(c_1) M_q^-(c_2) \cdots M_q^-(c_l)\right)|_{q=-\omega},$$
where $\begin{pmatrix}
J_\alpha(-\omega)&  *
\end{pmatrix}$ and $\begin{pmatrix}
-\omega&  1+\omega
\end{pmatrix}$ are $1 \times 2$ matrices, and $\cdot$ means the product of matrices. 
Easy calculation shows that  $\begin{pmatrix}
-\omega&  1+\omega
\end{pmatrix}=-\omega\begin{pmatrix}
1&  \omega
\end{pmatrix}$ and there exists $\zeta_i \in \{\pm 1, \pm \omega, \pm \omega^2 \}$ such that
$$\begin{pmatrix}
1 & \omega\end{pmatrix} \cdot  X_i = \zeta_i \begin{pmatrix}
1 & \omega\end{pmatrix} $$
for each $0 \le i \le 5$. So we can show \eqref{J(-w)} by induction on $l$.

\begin{rem}\label{S(-w)}
In the above notation, the matrix $X_3$ is not diagonalizable, and hence $X_3^n \ne E_2$ for all positive integers $n$. It means that the subgroup of $\mathsf{GL}(2,\CC)$ generated by $X_3$ is infinite, and hence $\{\cS_\alpha(-\omega) \mid \alpha \in \QQ \}$ is an infinite set.   
\end{rem}

\end{document}